\documentclass[10pt,twocolumn,amsmath,amssymb,aps,pra,secnumarabic,
    nofootinbib,groupedaddress]{revtex4-1}

\usepackage{mathptmx,amsthm,enumitem,tikz,comment}
\usepackage[colorlinks=true,linkcolor=red!70!black,citecolor=green!70!black,
    urlcolor=magenta!70!black,backref]{hyperref}
\usetikzlibrary{decorations.markings}
\usetikzlibrary{arrows}
\setlist[enumerate,1]{font=\bfseries,label=\arabic*.}

\makeatletter\def\@bibdataout@init{}\def\pre@bibdata{}\makeatother

\colorlet{darkred}{red!70!black}
\colorlet{darkblue}{blue!50!black}
\colorlet{medgreen}{green!70!black!70!white}

\newtheorem{theorem}{Theorem}[section]
\newtheorem{proposition}[theorem]{Proposition}
\newtheorem{lemma}[theorem]{Lemma}
\newtheorem{corollary}[theorem]{Corollary}

\newcommand{\ie}{\emph{i.e.}}
\newcommand{\Cf}{\emph{Cf.}}

\newcommand{\BPP}{\mathsf{BPP}}
\newcommand{\FP}{\mathsf{FP}}
\newcommand{\NPC}{\mathsf{NPC}}
\newcommand{\NP}{\mathsf{NP}}
\newcommand{\ccC}{\mathsf{C}}
\newcommand{\ccP}{\mathsf{P}}
\newcommand{\shPC}{\mathsf{\#PC}}
\newcommand{\shP}{\mathsf{\#P}}

\newcommand{\CSAT}{\mathsf{CSAT}}
\newcommand{\RSAT}{\mathsf{RSAT}}

\newcommand{\ZSAT}{\mathsf{ZSAT}}
\newcommand{\shCSAT}{\mathsf{\#CSAT}}
\newcommand{\shRSAT}{\mathsf{\#RSAT}}
\newcommand{\shZSAT}{\mathsf{\#ZSAT}}

\newcommand{\AD}{\operatorname{AD}}
\newcommand{\Alt}{\operatorname{Alt}}
\newcommand{\Aut}{\operatorname{Aut}}

\newcommand{\Inn}{\operatorname{Inn}}
\newcommand{\MCG}{\operatorname{MCG}}

\newcommand{\Out}{\operatorname{Out}}
\newcommand{\PSp}{\operatorname{PSp}}
\newcommand{\Rub}{\operatorname{Rub}}
\newcommand{\SU}{\operatorname{SU}}
\newcommand{\Sp}{\operatorname{Sp}}

\newcommand{\Sym}{\operatorname{Sym}}
\newcommand{\Tor}{\operatorname{Tor}}

\newcommand{\ab}{{\operatorname{ab}}}
\newcommand{\per}{{\operatorname{per}}}

\newcommand{\rep}{\operatorname{rep}}
\renewcommand{\wr}{\operatorname{wr}}
\newcommand{\bd}{\operatorname{bd}}
\newcommand{\sch}{\operatorname{sch}}

\newcommand{\onto}{\twoheadrightarrow}

\renewcommand{\setminus}{\smallsetminus}
\newcommand{\normaleq}{\unlhd}
\newcommand{\normal}{\lhd}

\newcommand{\cC}{\mathcal{C}}

\newcommand{\yes}{\mathrm{yes}}
\newcommand{\no}{\mathrm{no}}
\newcommand{\ceil}[1]{\lceil #1 \rceil}

\newcommand{\Z}{\mathbb{Z}}

\newcommand{\N}{\mathbb{N}}
\newcommand{\C}{\mathbb{C}}

\newcommand{\AND}{\mathrm{AND}}
\newcommand{\SWAP}{\mathrm{SWAP}}
\newcommand{\NOT}{\mathrm{NOT}}
\newcommand{\OR}{\mathrm{OR}}
\newcommand{\COPY}{\mathrm{COPY}}
\newcommand{\CNOT}{\mathrm{CNOT}}
\newcommand{\CCNOT}{\mathrm{CCNOT}}

\newcommand{\hR}{\hat{R}}

\newcommand{\tM}{\tilde{M}}

\newcommand{\cA}{\mathcal{A}}
\newcommand{\cB}{\mathcal{B}}

\newcommand{\vsubseteq}{\rotatebox{90}{$\subseteq$}}
\newcommand{\veq}{\rotatebox{90}{$=$}}
\newcommand{\vonto}{\rotatebox{270}{$\onto$}}
\newcommand{\defeq}{\stackrel{\mathrm{def}}=}

\newcommand{\Cor}[1]{Corollary~\ref{#1}}
\newcommand{\Fig}[1]{Figure~\ref{#1}}
\newcommand{\Sec}[1]{Section~\ref{#1}}
\newcommand{\Thm}[1]{Theorem~\ref{#1}}
\newcommand{\Lem}[1]{Lemma~\ref{#1}}

\newcommand{\eatline}{\vspace{-\baselineskip}}
\newenvironment{eq}[1]{\begin{equation}\label{#1}}
    {\end{equation}\ignorespacesafterend}

\begin{document}
\title{Computational complexity and 3-manifolds and zombies}
\author{Greg Kuperberg}
\email{greg@math.ucdavis.edu}
\thanks{Partly supported by NSF grant CCF-1319245}
\affiliation{University of California, Davis}

\author{Eric Samperton}
\email{egsamp@math.ucdavis.edu}
\thanks{Partly supported by NSF grant CCF-1319245}
\affiliation{University of California, Davis}

\date{\today}

\begin{abstract}
We show the problem of counting homomorphisms from the fundamental
group of a homology $3$-sphere $M$ to a finite, non-abelian simple
group $G$ is $\shP$-complete, in the case that $G$ is fixed and $M$ is
the computational input.  Similarly, deciding if there is a non-trivial
homomorphism is $\NP$-complete.  In both reductions, we can guarantee
that every non-trivial homomorphism is a surjection.  As a corollary, for
any fixed integer $m \ge 5$, it is $\NP$-complete to decide whether $M$
admits a connected $m$-sheeted covering.

Our construction is inspired by universality results in topological quantum
computation.  Given a classical reversible circuit $C$, we construct $M$
so that evaluations of $C$ with certain initialization and finalization
conditions correspond to homomorphisms $\pi_1(M) \to G$.  An intermediate
state of $C$ likewise corresponds to a homomorphism $\pi_1(\Sigma_g) \to
G$, where $\Sigma_g$ is a pointed Heegaard surface of $M$ of genus $g$.
We analyze the action on these homomorphisms by the pointed mapping class
group $\MCG_*(\Sigma_g)$ and its Torelli subgroup $\Tor_*(\Sigma_g)$.
By results of Dunfield-Thurston, the action of $\MCG_*(\Sigma_g)$ is as
large as possible when $g$ is sufficiently large; we can pass to the
Torelli group using the congruence subgroup property of $\Sp(2g,\Z)$.
Our results can be interpreted as a sharp classical universality property
of an associated combinatorial $(2+1)$-dimensional TQFT.
\end{abstract}

\maketitle

\section{Introduction}
\label{s:intro}

\subsection{Statement of Results}
\label{ss:results}

Given a finite group $G$ and a path-connected topological space $X$, let
\[ H(X,G) = \{f:\pi_1(X) \to G\} \]
be the set of homomorphisms from the fundamental group of $X$ to $G$.
Then the number $\#H(X,G) = |H(X,G)|$ is an important topological invariant
of $X$. For example, in the case that $X$ is a knot complement and $G =
\Sym(n)$ is a symmetric group, $\#H(X,G)$ was useful for compiling a table
of knots with up to 15 crossings \cite{Lickorish:intro}.  (We use both
notations $\#S$ and $|S|$ to denote the cardinality of a finite set $S$,
the former to emphasize algorithmic counting problems.)

Although these invariants can be powerful, our main result is that
they are often computationally intractable, assuming that $\ccP \neq \NP$.
We review certain considerations:
\begin{itemize}
\item We suppose that $X$ is given by a finite triangulation, as a reasonable
standard for computational input.

\item We are interested in the case that $|H(X,G)|$ is intractible because
of the choice of $X$ rather than the choice of $G$.  Therefore we fix $G$.
We are also more interested in the case when $H(X,J)$ is trivial for
every proper subgroup $J < G$.

\item If $G$ is abelian, then $\#H(X,G)$ is determined by the integral
homology group $H_1(X) = H_1(X;\Z)$; both can be computed in polynomial time.
We are thus more interested in the case that $H_1(X) = 0$ and $G$ is perfect,
in particular when $G$ is non-abelian simple.

\item If $X$ is a simplicial complex, or even an $n$-manifold with $n \ge
4$, then $\pi_1(X)$ can be any finitely presented group.   By contrast,
3-manifold groups are highly restricted.  We are more interested in the
case that $X = M$ is a 3-manifold.  If in addition $M$ is closed and $H_1(M)
= 0$, then $M$ is a homology 3-sphere.
\end{itemize}

To state our main result, we pass to the related invariant $\#Q(X,G) =
|Q(X,G)|$, where $Q(X,G)$ is the set of normal subgroups $\Gamma \normaleq
\pi_1(X)$ such that the quotient $\pi_1(X)/\Gamma$ is isomorphic to $G$.

\begin{theorem}  Let $G$ be a fixed, finite, non-abelian simple group.
If $M$ is a triangulated homology 3-sphere, then the invariant $\#Q(M,G)$
is $\shP$-complete via a parsimonious reduction.  The reduction also
guarantees that $\#Q(M,J) = 0$ for any non-trivial, proper subgroup $J < G$.
\label{th:main} \end{theorem}

\Sec{ss:classes} gives more precise definitions of the complexity theory
concepts in \Thm{th:main}.  Briefly, a counting problem is in $\shP$
if there is a polynomial-time algorithm to verify the objects being
counted; it is $\shP$-hard if it is as hard as any counting problem in
$\shP$; and it is $\shP$-complete if it is both in $\shP$ and $\shP$-hard.
A \emph{parsimonious reduction} from a counting problem $g$ to a counting
problem $f$ (to show that $f$ is as hard as $g$) is a mapping $h$, computable
in polynomial time, such that $g(x) = f(h(x))$.  This standard of hardness
tells us not only that $\#Q(M,G)$ is computationally intractible, but
also that any partial information from it (for instance, its parity) is
intractible; see \Thm{th:vv}.  An even stricter standard is a \emph{Levin
reduction}, which asks for a bijection between the objects being counted
that is computable in polynomial time (in both directions).  In fact, our
proof of \Thm{th:main} yields a Levin reduction from any problem in $\shP$
to the problem $\#Q(M,G)$.

The invariants $\#H(X,G)$ and $\#Q(X,G)$ are related by the following
equation:
\begin{eq}{e:homsum} |H(X,G)| = \sum_{J \leq G} |\Aut(J)|\cdot|Q(X,J)|.
\end{eq}
If $\pi_1(X)$ has no non-trivial surjections to any simple group smaller
than $G$, as \Thm{th:main} can provide, then
\begin{eq}{e:hq} |H(X,G)| = |\Aut(G)|\cdot|Q(X,G)| + 1. \end{eq}
Thus we can say that $\#H(M,G)$ is \emph{almost parsimoniously}
$\shP$-complete for homology 3-spheres.  It is parsimonious except for the
trivial homomorphism and up to automorphisms of $G$, which are both minor,
unavoidable corrections.  This concept appears elsewhere in complexity
theory; for instance, the number of 3-colorings of a planar graph is almost
parsimoniously $\shP$-complete \cite{Barbanchon:unique}.

In particular, the fact that $\#Q(M,G)$ is parsimoniously $\shP$-hard
implies that existence is Karp $\NP$-hardness (again see \Sec{ss:classes}).
Thus \Thm{th:main} has the following corollary.

\begin{corollary} Let $G$ be a fixed, finite, non-abelian simple group, and
let $M$ be a triangulated homology 3-sphere regarded as computational input.
Then it is Karp $\NP$-complete to decide whether there is a non-trivial
homomorphism $f:\pi_1(M) \to G$, even with the promise that every such
homomorphism is surjective.
\label{c:np} \end{corollary}

\Cor{c:np} in turn has a corollary concerning connected covering spaces.
In the proof of the corollary and later in the paper, we let $\Sym(n)$
be the symmetric group and $\Alt(n)$ be the alternating group, both acting
on $n$ letters.

\begin{corollary} For each fixed $n \ge 5$, it is $\NP$-complete to decide
whether a homology 3-sphere $M$ has a connected $n$-sheeted cover, even
with the promise that it has no connected $k$-sheeted cover with $1 < k < n$.
\label{c:covers} \end{corollary}

\begin{proof} Recall that $\Alt(n)$ is simple when $n \ge 5$.
The $n$-sheeted covers $\tM$ of $M$ are bijective with homomorphisms
$f:\pi_1(M) \to \Sym(n)$, considered up to conjugation in $\Sym(n)$.
If $M$ is a homology 3-sphere, then $\pi_1(M)$ is a perfect group and we can
replace $\Sym(n)$ by $\Alt(n)$. If $\tM$ is disconnected, then $f$ does not
surject onto $\Alt(n)$.  Thus, we can apply \Cor{c:np} with $G = \Alt(n)$.
\end{proof}

\begin{figure}[htb]
\frame{\includegraphics[width=.75\columnwidth]{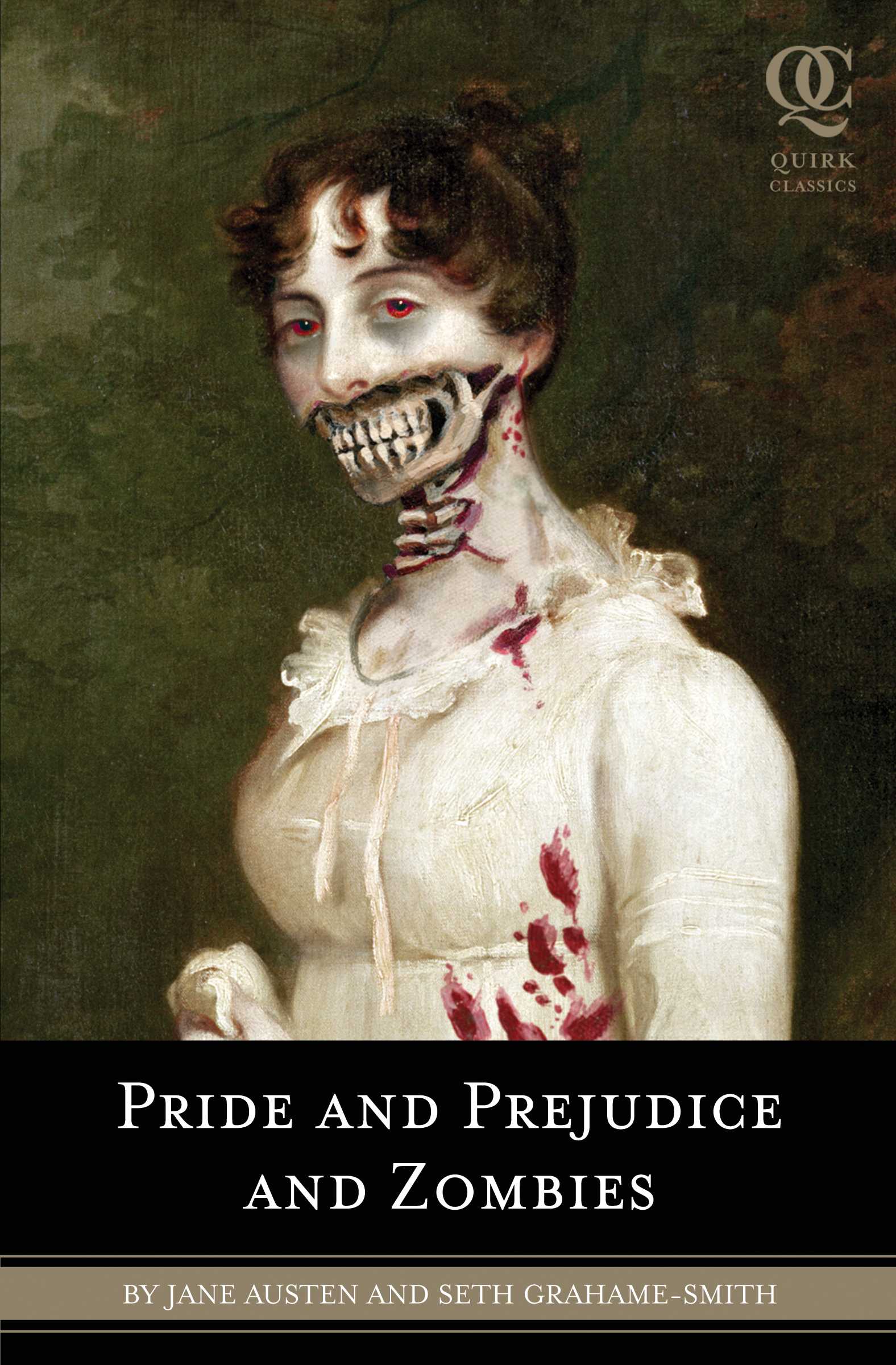}}
\caption{A classic novel, with zombies \cite{AG:zombies}.}
\label{f:ppzombies} \end{figure}

The idea of our proof of \Thm{th:main} is as follows.  Let $\Sigma_g$
be a standard oriented surface of genus $g$ with a marked basepoint, and
let $G$ be a (not necessarily simple) finite group.  Then we can interpret
the set of homomorphisms, or \emph{representation set},
\[ \hR_g(G) \defeq H(\Sigma_g,G) = \{f:\pi_1(\Sigma_g) \to G\} \]
as roughly the set of states of a computer memory.  We can interpret a word
in a fixed generating set of the pointed, oriented mapping class group
$\MCG_*(\Sigma_g)$ as a reversible digital circuit acting on $\hR_g(G)$,
the set of memory states.  Every closed, oriented 3-manifold $M$ can be
constructed as two handlebodies $(H_g)_I$ and $(H_g)_F$ that are glued
together by an element $\phi \in \MCG_*(\Sigma_g)$.  We can interpret
$\phi$ as a reversible digital circuit in which the handlebodies partially
constrain the input and output.

To understand the possible effect of $\phi$, we want to decompose $\hR_g(G)$
into $\MCG_*(\Sigma_g)$-invariant subsets.  The obvious invariant of $f
\in \hR_g(G)$ is its image $f(\pi_1(\Sigma_g)) \leq G$; to account for it,
we first restrict attention to the subset
\[ R_g(G) \defeq \{f: \pi_1(\Sigma_g) \onto G\} \subseteq \hR_g(G) \]
consisting of surjective homomorphisms.

We must also consider a less obvious invariant.  Let $BG$ be the classifying
space of $G$, and recall that the group homology $H_*(G) = H_*(G;\Z)$
can be defined as the topological homology $H_*(BG)$.  Recall that a
homomorphism $f:\pi_1(\Sigma_g) \to G$ corresponds to a map $f:\Sigma_g
\to BG$ which is unique up to pointed homotopy.  Every $f \in \hR_g(G)$
then yields a homology class
\[ \sch(f) \defeq f_*([\Sigma_g]) \in H_2(G), \]
which we call the \emph{Schur invariant} of $f$; it is
$\MCG_*(\Sigma_g)$-invariant.  Given $s \in H_2(G)$, the subset
\[ R_g^s(G) \defeq \{f \in R_g \mid \sch(f) = s\} \]
is then also $\MCG_*(\Sigma_g)$-invariant.  Note that $\sch(f)$ is not always
$\Aut(G)$-invariant because $\Aut(G)$ may act non-trivially on $H_2(G)$.
Fortunately, $R_g^0(G)$ is always $\Aut(G)$-invariant.  We summarize the
relevant results of Dunfield-Thurston in the following theorem.

\begin{theorem}[Dunfield-Thurston {\cite[Thms. 6.23 \& 7.4]{DT:random}}] Let
$G$ be a finite group.
\begin{enumerate}
\item For every sufficiently large $g$ (depending on $G$), $\MCG_*(\Sigma_g)$
acts transitively on $R^s_g(G)$ for every $s \in H_2(G)$.
\item If $G$ is non-abelian and simple, then for every sufficiently large
$g$, the image of the action of $\MCG_*(\Sigma_g)$ on $R^0_g(G)/\Aut(G)$
is $\Alt(R^0_g(G)/\Aut(G))$.
\end{enumerate}
\label{th:dt} \end{theorem}

To make effective use of \Thm{th:dt}, we strengthen its second part in
three ways to obtain \Thm{th:refine}.   First, \Thm{th:dt} holds for the
pointed Torelli group $\Tor_*(\Sigma_g)$.  Second, we define an analogue
of alternating groups for $G$-sets that we call \emph{Rubik groups}, and
we establish \Thm{th:rubik}, a non-trivial structure theorem to generate a
Rubik group.  \Thm{th:refine} gives a lift of the image of $\MCG_*(\Sigma_g)$
from $\Alt(R^0_g(G)/\Aut(G))$ to the Rubik group $\Rub_{\Aut(G)}(R^0_g(G))$.
Third, we still obtain the image $\Rub_{\Aut(G)}(R^0_g(G))$ even if
we restrict to the subgroup of $\Tor_*(\Sigma_g)$ that pointwise fixes
$\hR_g(G) \setminus R_g(G)$, the set of non-surjective homomorphisms.

As a warm-up for our proof of \Thm{th:main}, we can fix $g$, and try
to interpret
\[ A = R_g^0(G)/\Aut(G) \]
as a computational alphabet.  If $g$ is large enough, then we can apply
\Thm{th:dt} to $R^0_{2g}(G)$ to obtain a universal set of reversible binary
gates that act on $A^2 \subset R^0_{2g}(G)/\Aut(G)^2$, implemented as mapping
class elements or \emph{gadgets}.  (A gadget in computational complexity is
an informal concept that refers to a combinatorial component of a complexity
reduction.)  The result can be related to a certain reversible circu
I'm not a scientist. However, I am well-read and take pride in reading both sides of most issues.it
model $\RSAT_{A,I,F}$.  (See \Sec{ss:circuits}.  The $\shP$-hardness of
$\RSAT$, established in \Thm{th:rsat}, is a standard result but still takes
significant work.)  We can convert a reversible circuit of width $n$ to an
element $\phi \in \MCG_*(\Sigma_{ng})$ that acts on $A^n$, and then make
$M$ from $\phi$.  In this way, we can reduce $\shRSAT_{A,I,F}$ to $\#Q(M,G)$.

For our actual reduction, we will need to take steps to address three
issues, which correspond to the three ways that \Thm{th:refine} is sharper
than \Thm{th:dt}.
\begin{itemize}
\item We want the larger calculation in $\hR_{ng}(G)$ to avoid symbols in
$\hR_g(G) \setminus R^0_g(G)$ that could contribute to $\#Q(M,G)$.

\item We want a parsimonious reduction to $\#Q(M,G)$, which means that we
must work with $R^0_g(G)$ rather than its quotient $A$.

\item Mapping class gadgets should be elements of the Torelli group,
to guarantee that $M$ is a homology 3-sphere.
\end{itemize}

To address the first issue: We can avoid states in $R^s_g(G)$ with
$s \ne 0$ because, if a surface group homomorphism $f:\pi_1(\Sigma_g)
\onto G$ has $\sch(f) \ne 0$, then it cannot extend over a handlebody.
If $f(G)$ has a non-trivial abelianization, then the fact that we will
produce a homology 3-sphere will kill its participation.  If $f$ is not
surjective but $f(G)$ is perfect, then we will handle this case by acting
trivially on $R_g(K)$ for a simple quotient $K$ of $f(G)$.  The trivial
homomorphism $z \in \hR_g(G)$ is particularly problematic because it cannot
be eliminated using the same techniques; we call it the \emph{zombie symbol}.
(See also \Fig{f:ppzombies} for another source of our ideas.)  We define
an ad hoc reversible circuit model, $\ZSAT$, that has zombie symbols.
We reduce $\RSAT$ to $\ZSAT$ by converting the zombie symbols to warning
symbols that do not finalize, unless all of the symbols are zombies.
The full construction, given in Lemmas \ref{l:zsat} and \ref{l:glue},
is more complicated because these steps must be implemented with binary
gates in $\MCG_*(\Sigma_{2g})$ rather than unary gates in $\MCG_*(\Sigma_g)$.

To address the second issue: A direct application of \Thm{th:dt} would
yield a factor of $|\Aut(G)|^n$ in the reduction from $\shRSAT_{A,I,F}$ to
$\#H(M,G)$, when the input is a reversible circuit of width $n$.  We want
to reduce this to a single factor of $|\Aut(G)|$ in order to construct a
parsimonious reduction to $\#Q(M,G)$.  The $\ZSAT$ model also has an action
of $K = \Aut(G)$ on its alphabet to model this.  \Lem{l:zsat} addresses
the problem by relying on the Rubik group refinement in \Thm{th:refine},
and by creating more warning symbols when symbols are misaligned relative
to the group action.

To ensure that the resulting manifold is a homology 3-sphere, we
implement gates in the pointed Torelli subgroup $\Tor_*(\Sigma_g)$ of
$\MCG_*(\Sigma_g)$.  This is addressed in \Thm{th:refine}.  Recall that
$\Tor_*(\Sigma_g)$ is the kernel of the surjective homomorphism
\[ f:\MCG_*(\Sigma_g) \to H_1(\Sigma_g) \cong \Sp(2g,\Z) \]
where $H_1(\Sigma_g)$ is equipped with its integral symplectic intersection
form.  The proof of \Thm{th:refine} uses rigidity properties of $\Sp(2g,\Z)$
combined with Goursat's lemma (\Lem{l:goursat}).

\begin{figure}[htb]
\[ \CSAT \xrightarrow{\text{Sec. \ref{ss:circuits}}} \RSAT
    \xrightarrow{\text{Sec. \ref{ss:zombies}}} \ZSAT
    \xrightarrow{\text{Sec. \ref{ss:mcg}}} \#Q(M,G) \]
\caption{The reductions in the proof of \Thm{th:main}.}
\label{f:reductions} \end{figure}

\Fig{f:reductions} summarizes the main reductions in the proof of
\Thm{th:main}.

\subsection{Related work}
\label{ss:related}

As far as we know, the closest prior result to our \Thm{th:main} is due to
Krovi and Russell \cite{KR:finite}.  Given a link $L \subseteq S^3$, they
consider a refinement $\#H(S^3 \setminus L,G,C)$ of $\#H(S^3 \setminus L,G)$
in which they only count the group homomorphisms that send the meridian
elements of $\pi_1(S^3 \setminus L)$ to a specific conjugacy class $C
\subseteq \Alt(m)$.  They show that the exact value is $\shP$-complete
when $m \ge 5$, but they do not obtain a parsimonious reduction.  Instead,
they retain an exponentially small error term.  In particular, they do
not obtain $\NP$-hardness for the existence problem.  However, in their
favor, we found it easier to prove \Thm{th:main} in the case of closed
3-manifolds than in the case of link complements, which we will address in
future work \cite{K:coloring}.

\begin{figure}[htb]
\begin{tabular}{l|c|c|c|c|c|c}
& \multicolumn{2}{|c}{one equation} &
    \multicolumn{2}{|c}{equations} &
    \multicolumn{2}{|c}{homomorphisms} \\ \hline
finite target $G$ & $\exists$ \cite{GR:solving} & \#  \cite{NJ:counting}
    & $\exists$ \cite{GR:solving} & \#  \cite{NJ:counting}
    & $\exists$ & \# \\ \hline
abelian & $\ccP$ & $\FP$ & $\ccP$ & $\FP$ & $\ccP$ & $\FP$ \\
nilpotent & $\ccP$ & ? & $\NPC$ & $\shPC$ & ? & ? \\
solvable & ?& ? & $\NPC$ & $\shPC$ & ? & ? \\
non-solvable & $\NPC$ & $\shPC$ & $\NPC$ & $\shPC$ & ? & ? \\
non-ab. simple & $\NPC$ & $\shPC$ & $\NPC$ & $\shPC$ & $\NPC$! & $\shPC$!
\end{tabular}
\caption{The complexity of solving equations over or finding homomorphisms
to a fixed finite target group $G$.  Here $\ccP$ denotes polynomial time for
a decision problem, $\FP$ denotes polynomial time for a function problem,
$\NPC$ is the class of $\NP$-complete problems, and $\shPC$ is the class of
$\shP$-complete problems.  Exclamation marks indicate results in this paper.}
\label{f:knownunknowns}
\end{figure}

We can also place \Thm{th:main} in the context of other counting
problems involving finite groups.   We summarize what is known in
\Fig{f:knownunknowns}.  Given a finite group $G$, the most general analogous
counting problem is the number of solutions to a system of equations that
may allow constant elements of $G$ as well as variables.  Nordh and Jonsson
\cite{NJ:counting} showed that this problem is $\shP$-complete if and only
if $G$ is non-abelian, while Goldman and Russell \cite{GR:solving} showed
that the existence problem is $\NP$-complete.  If $G$ is abelian, then any
finite system of equations can be solved by the Smith normal form algorithm.
These authors also considered the complexity of a single equation.  In this
case, the existence problem has unknown complexity if $G$ is solvable but
not nilpotent, while the counting problem has unknown complexity if $G$
is solvable but not abelian.

If all of the constants in a system of equations over $G$ are set to $1
\in G$, then solving the equations amounts to finding group homomorphisms
$f:\Gamma \to G$ from the finitely presented group $\Gamma$ given by the
equations.  By slight abuse of notation, we can call this counting problem
$\#H(\Gamma,G)$.  This is equivalent to the topological invariant $\#H(X,G)$
when $X$ is a simplicial complex, or even a triangulated $n$-manifold for
any fixed $n \ge 4$; in this case, given any finitely presented $\Gamma$,
we can construct $X$ with $\Gamma = \pi_1(X)$ in polynomial time.  To our
knowledge, \Thm{th:main} is a new result for the invariant $\#H(\Gamma,G)$,
even though we specifically construct $\Gamma$ to be a 3-manifold group
rather than a general finitely presented group.  For comparison,
both the non-triviality problem and the word problem are as difficult as
the halting problem for general $\Gamma$ \cite{Poonen:sampler}, while the
word problem and the isomorphism problem are both recursive for 3-manifold
groups \cite{AFW:decision,K:homeo3}.

In the other direction, if $M$ is a closed 2-manifold, then
there are well known formulas of Frobenius-Schur and Mednykh for
$\#H(M,G)$ \cite{FS:gruppen,Mednykh:compact,FQ:finite} for any
finite group $G$ as a function of the genus and orientability of $M$
\cite{FS:gruppen,Mednykh:compact,FQ:finite}.  Mednykh's formula was
generalized by Chen \cite{Chen:seifert} to the case of Seifert-fibered
3-manifolds.  In \Sec{ss:standard}, we give a generalization of these
formulas to the class of bounded-width simplicial complexes.

Our approach to \Thm{th:main} (and that of Krovi and Russell for
their results) is inspired by quantum computation and topological
quantum field theory.  Every unitary modular tensor category (UMTC)
$\cC$ yields a unitary 3-dimensional topological quantum field theory
\cite{RT:ribbon,RT:manifolds,Turaev:quantum}.  The topological quantum
field theory assigns a vector space $V(\Sigma_g)$, or \emph{state
space}, to every oriented, closed surface.  It also assigns a state
space $V(\Sigma_{g,n},C)$ to every oriented, closed surface with $n$
boundary circles, where $C$ is an object in $\cC$ interpreted as the
``color" of each boundary circle.  Each state space $V(\Sigma_{g,n},C)$
has a projective action of the mapping class group $\MCG_*(\Sigma_{g,n})$.
(In fact the unpointed mapping class group $\MCG(\Sigma_{g,n})$ acts, but
we will keep the basepoint for convenience.)  These mapping class group
actions then extend to invariants of 3-manifolds and links in 3-manifolds.

Finally, the UMTC $\cC$ is universal for quantum computation if the image of
the mapping class group action on suitable choices of $V(\Sigma_{g,n},C)$
is large enough to simulate quantum circuits on $m$ qubits, with $g,n =
O(m)$.  If the action is only large enough to simulate classical circuits
on $m$ bits, then it is still classically universal.  These universality
results are important for the fault-tolerance problem in quantum computation
\cite{FLW:universal,K:tvcodes}.

One early, important UMTC is the (truncated) category $\rep_q(\SU(2))$
of quantum representations of $\SU(2)$ at a principal root of unity.  This
category yields the Jones polynomial for a link $L \subseteq S^3$ (taking $C
= V_1$, the first irreducible object) and the Jones-Witten-Reshetikhin-Turaev
invariant of a closed 3-manifold.  In separate papers, Freedman, Larsen,
and Wang showed that $V(\Sigma_{0,n},V_1)$ and $V(\Sigma_{g,0})$ are
both quantumly universal representations of $\MCG_*(\Sigma_{0,n})$ and
$\MCG_*(\Sigma_{g,0})$ \cite{FLW:universal,FLW:two}.

Universality also implies that any approximation of these invariants
that could be useful for computational topology is $\shP$-hard
\cite{K:jones,AA:hardness}.  Note that exact evaluation of the Jones
polynomial was earlier shown to be $\shP$-hard without quantum computation
methods \cite{JVW:complexity}.

If $G$ is a finite group, then the invariant $\#H(M,G)$ for a 3-manifold
$M$ also comes from a UMTC, namely the categorical double $D(\rep(G))$
of $\rep(G)$, that was treated (and generalized) by Dijkgraaf and Witten
and others \cite{K:hopf,DW:group,FQ:finite}.  In this case, the state
space $V(\Sigma_{g,0})$ is the vector space $\C[\hR_g(G)/\Inn(G)]$, and
the action of $\MCG_*(\Sigma_{g,0})$ on $V(\Sigma_{g,0})$ is induced by
its action on $\hR_g(G)$.  Some of the objects in $D(\rep(G))$ are given by
conjugacy classes $C \subseteq G$, and the representation of the braid group
$\MCG_*(\Sigma_{0,n})$ with braid strands colored by a conjugacy class $C$
yields the invariant $\#H(S^3 \setminus L,G,C)$ considered by Krovi and
Russell.  Motivated by the fault tolerance problem, Ogburn and Preskill
\cite{OP:topological} found that the braid group action for $G = \Alt(5)$
is classically universal (with $C$ the conjugacy class of 3-cycles) and they
reported that Kitaev showed the same thing for $\Sym(5)$.   They also showed
if these actions are enhanced by quantum measurements in a natural sense,
then they become quantumly universal.  Later Mochon \cite{Mochon:finite}
extended this result to any non-solvable finite group $G$.  In particular,
he proved that the action of $\MCG_*(\Sigma_{0,n})$ is classically universal
for a suitably chosen conjugacy class $C$.

Mochon's result is evidence, but not proof, that $\#H(S^3 \setminus L,G,C)$
is $\shP$-complete for every fixed, non-solvable $G$ and every suitable
conjugacy class $C \subseteq G$ that satisfies his theorem.  His result
implies that if we constrain the associated braid group action with arbitrary
initialization and finalization conditions, then counting the number of
solutions to the constraints is parsimoniously $\shP$-complete.  However,
if we use a braid to describe a link, for instance with a plat presentation
\cite{K:jones}, then the description yields specific initialization and
finalizations conditions that must be handled algorithmically to obtain
hardness results.  Recall that in our proof of \Thm{th:main}, the state in
$\hR_g(G)$ is initialized and finalized using the handlebodies $(H_g)_I$
and $(H_g)_F$.  If we could choose any initialization and finalization
conditions whatsoever, then it would be much easier to establish (weakly
parsimonious) $\shP$-hardness; it would take little more work than to
cite \Thm{th:dt}.

\section{Complexity and algorithms}
\label{s:complexity}

\subsection{Complexity classes}
\label{ss:classes}

For background on the material in this section, and some of the treatment
in the next section as well, see Arora and Barak \cite{AB:modern} and the
Complexity Zoo \cite{zoo}.

Let $A$ be a finite \emph{alphabet} (a finite set with at least 2 elements)
whose elements are called \emph{symbols}, and let $A^*$ be the set of finite
words in $A$.   We can consider three kinds of computational problems with
input in $A^*$: decision problems $d$, counting problems $c$, and function
problems $f$, which have the respective forms
\begin{eq}{e:dcf} d:A^* \to \{\yes,\no\} \qquad c:A^* \to \N \qquad
    f:A^* \to A^*. \end{eq}
The output set of a decision problem can also be identified with the
Boolean alphabet
\[ A = \Z/2 = \{1,0\} \cong \{\text{true},\text{false}\}
    \cong \{\yes,\no\}. \]

A \emph{complexity class} $\ccC$ is any set of function, counting, or
decision problems, which may either be defined on all of $A^*$ or require
a promise.  A specific, interesting complexity class is typically defined as
the set of all problems that can be computed with particular computational
resources.  For instance, $\ccP$ is the complexity class of all decision
problems $d$ such that $d(x)$ can be computed in polynomial time (in the
length $|x|$ of the input $x$) by a deterministic Turing machine.  $\FP$
is the analogous class of function problems that are also computable in
polynomial time.

A \emph{promise problem} is a function $d$, $c$, or $f$ of the same
form as \eqref{e:dcf}, except whose domain can be an arbitrary subset
$S \subseteq A^*$.  The interpretation is that an algorithm to compute
a promise problem can accept any $x \in A^*$ as input, but its output is
only taken to be meaningful when it is promised that $x \in S$.

The input to a computational problem is typically a data type such as an
integer, a finite graph, a simplicial complex, etc.  If such a data type can
be encoded in $A^*$ in some standard way, and if different standard encodings
are interconvertible in $\FP$, then the encoding can be left unspecified.
For instance, the decision problem of whether a finite graph is connected is
easily seen to be in $\ccP$; the specific graph encoding is not important.
Similarly, there are various standard encodings of the non-negative integers
$\N$ in $A^*$.  Using any such encoding, we can also interpret $\FP$
as the class of counting problems that can be computed in polynomial time.

The complexity class $\NP$ is the set of all decision problems $d$ that can
be answered in polynomial time with the aid of a prover who wants to convince
the algorithm (or verifier) that the answer is ``yes".  In other words, every
$d \in \NP$ is given by a two-variable predicate $v \in \ccP$.  Given an input
$x$, the prover provides a witness $y$ whose length $|y|$ is some polynomial
in $|x|$.  Then the verifier computes $v(x,y)$, with the conclusion that
$d(x) = \yes$ if and only if there exists $y$ such that $v(x,y) = \yes$.
The witness $y$ is also called a \emph{proof} or \emph{certificate},
and the verification $v$ is also called a \emph{predicate}.  Likewise,
a function $c(x)$ is in $\shP$ when it is given by a predicate $v(x,y)$;
in this case $c(x)$ is the number of witnesses $y$ that satisfy $v(x,y)$.
For instance, whether a finite graph $G$ (encoded as $x$) has a 3-coloring
is in $\NP$, while the number of 3-colorings of $G$ is in $\shP$.  In both
cases, a 3-coloring of $G$ serves as a witness $y$.

A computational problem $f$ may be $\NP$-hard or $\shP$-hard with the
intuitive meaning that it is provably at least as difficult as any problem
in $\NP$ or $\shP$.   A more rigorous treatment leads to several different
standards of hardness.  One quite strict standard is that any problem $g$
in $\NP$ or $\shP$ can be reduced to the problem $f$ by converting the
input; \ie, there exists $h \in \FP$ such that
\[ g(x) = f(h(x)). \]
If $f, g \in \NP$, then this is called \emph{Karp reduction}; if $f, g \in
\shP$, then it is called \emph{parsimonious reduction}.  Evidently, if a
counting problem $c$ is parsimoniously $\shP$-hard, then the corresponding
existence problem $d$ is Karp $\NP$-hard.

When a problem $f$ is $\shP$-hard by some more relaxed standard than
parsimonious reduction, there could still be an algorithm to obtain
some partial information about the value $f$, such as a congruence or an
approximation, even if the exact value is intractible.   For instance,
the permanent of an integer matrix is well-known to be $\shP$-hard
\cite{Valiant:permanent}, but its parity is the same as that of the
determinant, which can be computed in polynomial time.  However, when
a counting problem $c$ is parsimoniously $\shP$-hard, then the standard
conjecture that $\NP \not\subseteq \BPP$ implies that it is intractible to
obtain any partial information about $c$.  Here $\BPP$ is the set of problems
solvable in randomized polynomial time with a probably correct answer.

\begin{theorem}[Corollary of Valiant-Vazirani \cite{VV:unique}]  Let $c$
be a parsimoniously $\shP$-hard problem, and let $b > a \ge 0$ be distinct,
positive integers.  Then it is $\NP$-hard to distinguish $c(x) = a$ from
$c(x) = b$ via a Cook reduction in $\BPP$, given the promise that $c(x)
\in \{a,b\}$.
\label{th:vv} \end{theorem}

When we say that an algorithm $\cA$ obtains partial information about
the value of $c(x)$, we mean that it can calculate $f(c(x))$ for some
non-constant function $f$.  Thus it can distinguish some pair of cases
$c(x) = a$ and $c(x) = b$; and by \Thm{th:vv}, this is $\NP$-hard.  Here a
\emph{Cook reduction} is a polynomial-time algorithm $\cB$ (in this case
randomized polynomial time) that can call $\cA$ as a subroutine.

\begin{proof} Given a problem $d \in \NP$, Valiant and Vazirani construct
a randomized algorithm $\cB$ that calculates $d(x)$ using a collection of
predicates $v_1(x,y)$ in $\ccP$ that usually have at most one solution in $y$.
Thus, if an algorithm $\cA$ can solve each problem
\[ d_1(x) = \exists? y \text{\ such that\ } v_1(x,y) = \yes \]
under the promise that at most one $y$ exists, then $\cA$ can be used as
a subroutine to compute the original $d$.  Such a predicate $v_1(x,y)$
may occasionally have more than one solution, but this happens rarely and
still allows $\cB$ to calculate $d$ by the standard that its output only
needs to be probably correct.

Given such a predicate $v_1(x,y)$, it is easy to construct another
predicate $v_2(x,y)$ in $\ccP$ that has $b-a$ solutions in $y$ for each
solution to $v_1(x,y)$, and that has $a$ other solutions in $y$ regardless.
Thus $v_2(x,y)$ has $b$ solutions when $d_1(x) = \yes$ and $a$ solutions
when $d_1(x) = \no$.  Thus, an algorithm $\cA$ that can distinguish
$c(x) = a$ from $c(x) = b$ can be used to calculate $d_1(x)$, and by the
Valiant-Vazirani construction can be used to calculate $d(x)$.
\end{proof}

A decision problem $d$ which is both in $\NP$ and $\NP$-hard is called
\emph{$\NP$-complete}, while a counting problem which is both in $\shP$
and $\shP$-hard is called \emph{$\shP$-complete}.   For instance the
decision problem $\CSAT$, circuit satisfiability over an alphabet $A$, is
Karp $\NP$-complete, while the counting version $\shCSAT$ is parsimoniously
$\shP$-complete (\Thm{th:csat}).  Thus, we can prove that any other problem
is $\NP$-hard by reducing $\CSAT$ to it, or $\shP$-hard by reducing $\shCSAT$
to it.

We mention three variations of parsimonious reduction.   A counting function
$c$ is \emph{weakly parsimoniously} $\shP$-hard if for every $b \in \shP$,
there are $f,g \in \FP$ such that
\[ b(x) = f(c(g(x)),x). \]
The function $c$ is \emph{almost parsimoniously} $\shP$-hard if $f$ does
not depend on $x$, only on $c(g(x))$.  In either case, we can also ask for
$f(c,x)$ to be 1-to-1 on the set of values of $c$ with $f^{-1} \in \FP$,
linear or affine linear in $c$, etc.  So, \Thm{th:main} says that $\#H(M,G)$
is almost parsimoniously $\shP$-complete.

Finally, suppose that $c(x)$ counts the number of solutions to $v(x,y)$
and $b(x)$ counts the number of solutions to $u(x,y)$.  Then a \emph{Levin
reduction} is a map $h \in \FP$ and a bijection $f$ with $f,f^{-1} \in \FP$
such that
\[ u(x,y) = v(h(x),f(y)). \]
Obviously Levin reduction implies parsimonious reduction.

\subsection{Circuits}
\label{ss:circuits}

Given an alphabet $A$, a \emph{gate} is a function $\alpha:A^k \to A^\ell$.
A \emph{gate set} $\Gamma$ is a finite set of gates, possibly with
varying sizes of domain and target, and a \emph{circuit} over $\Gamma$ is a
composition of gates in $\Gamma$ in the pattern of a directed, acyclic graph.
A gate set $\Gamma$ is \emph{universal} if every function $f:A^n \to A^m$
has a circuit.  For example, if $A = \Z/2$, then the gate set
\[ \Gamma =\{\AND, \OR, \NOT, \COPY\} \]
is universal, where $\AND$, $\OR$, and $\NOT$ are the standard Boolean
operations and the $\COPY$ gate is the diagonal embedding $a \mapsto (a,a)$.

Let $A$ be an alphabet with a universal gate set $\Gamma$, and suppose that
$A$ has a distinguished symbol $\yes \in A$.  Choose a standard algorithm
to convert an input string $x \in A^*$ to a circuit $C_x$ with one output.
Then the \emph{circuit satisfiability problem} $\CSAT_{A,\Gamma}(x)$ asks
whether the circuit $C_x$ has an input $y$ such that $C_x(y) = \yes$.
It is not hard to construct a Levin reduction of $\CSAT_{A,\Gamma}$ from
any one alphabet and gate set to any other, so we can just call any such
problem $\CSAT$.  $\CSAT$ also has an obvious counting version $\shCSAT$.

\begin{theorem}[Cook-Levin-Karp] $\CSAT$ is Karp $\NP$-complete and $\shCSAT$
is parsimoniously $\shP$-complete.
\label{th:csat} \end{theorem}

(See Arora-Barak \cite[Sec. 6.1.2 \& Thm. 17.10]{AB:modern} for a proof of
\Thm{th:csat}.)

We will need two variations of the circuit model that still satisfy
\Thm{th:csat}:  Reversible circuits and planar circuits.

A \emph{reversible circuit} \cite{FT:logic} is a circuit $C$ in which
every gate $\alpha:A^k \to A^k$ in the gate set $\Gamma$ is a bijection;
thus the evaluation of $C$ is also a bijection.  We say that $\Gamma$ is
\emph{reversibly universal} if for any sufficiently large $n$, the gates of
$\Gamma$ in different positions generate either $\Alt(A^n)$ or $\Sym(A^n)$.
(If $|A|$ is even, then we cannot generate any odd permutations when $n$
is larger than the size of any one gate in $\Gamma$.)

\begin{figure}[htb] \begin{center}
\begin{tikzpicture}[scale=.7,semithick,decoration={markings,
    mark=at position 0.08 with {\arrow{angle 90}},
    mark=at position 0.94 with {\arrow{angle 90}}}]
\draw (0,0) node[anchor=east] {$x_5$};
\draw[postaction={decorate}] (0,0) -- (7,0);
\draw (7,0) node[anchor=west] {$y_5$};
\draw (0,1) node[anchor=east] {$x_4$};
\draw[postaction={decorate}] (0,1) -- (7,1);
\draw (7,1) node[anchor=west] {$y_4$};
\draw (0,2) node[anchor=east] {$x_3$};
\draw[postaction={decorate}] (0,2) -- (7,2);
\draw (7,2) node[anchor=west] {$y_3$};
\draw (0,3) node[anchor=east] {$x_2$};
\draw[postaction={decorate}] (0,3) -- (7,3);
\draw (7,3) node[anchor=west] {$y_2$};
\draw (0,4) node[anchor=east] {$x_1$};
\draw[postaction={decorate}] (0,4) -- (7,4);
\draw (7,4) node[anchor=west] {$y_1$};
\draw[fill=white] (1.1,-.35) rectangle (1.9,1.35);
\draw (1.5,.5) node {$\alpha_2$};
\draw[fill=white] (1.1,2.65) rectangle (1.9,4.35);
\draw (1.5,3.5) node {$\alpha_1$};
\draw[fill=white] (3.1,.65) rectangle (3.9,2.35);
\draw (3.5,1.5) node {$\alpha_3$};
\draw[fill=white] (5.1,1.65) rectangle (5.9,3.35);
\draw (5.5,2.5) node {$\alpha_4$};
\draw (3.5,-.5) node[anchor=north] {$C(x) = y$};
\end{tikzpicture} \end{center}
\caption{A planar, reversible circuit.}
\label{f:planar} \end{figure}
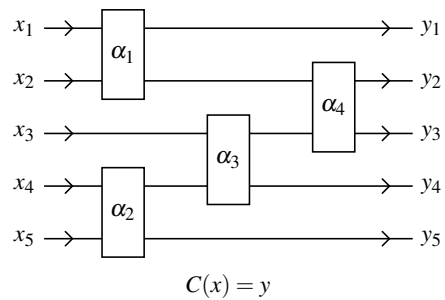

A circuit $C$ is \emph{planar} if its graph is a planar graph placed in a
rectangle in the plane, with the inputs on one edge and the output on an
opposite edge.  The definition of a universal gate for general circuits
carries over to planar circuits; likewise the definition for reversible
circuits carries over to reversible planar circuits.  (See \Fig{f:planar}.)
We can make a circuit or a reversible circuit planar using reversible
$\SWAP$ gates that take $(a,b)$ to $(b,a)$.  Likewise, any universal gate
set becomes planar-universal by adding the $\SWAP$ gate.  Thus, the planar
circuit model is equivalent to the general circuit model.

The reduction from general circuits to reversible circuits is more
complicated.

\begin{lemma} Let $A$ be an alphabet for reversible circuits.
\begin{enumerate}
\item[1.] If $|A| \ge 3$, then $\Gamma = \Alt(A^2)$ is a universal set of
binary gates.
\item[2.] If $|A| = 2$, then $\Gamma = \Alt(A^3)$ is a universal set
of ternary gates.
\item[3.] If $|A|$ is even, then $\Sym(A^n) \subseteq \Alt(A^{n+1})$.
\end{enumerate}
\label{l:rev} \end{lemma}

Different versions of \Lem{l:rev} are standard in the reversible
circuit literature.  For instance, when $A = \Z/2$, the foundational
paper \cite{FT:logic} defines the Fredkin gate and the Toffoli gate,
each of which is universal together with the $\NOT$ gate.  Nonetheless,
we did not find a proof for all values of $|A|$, so we give one.

\begin{proof} Case 3 of the lemma is elementary, so we
concentrate on cases 1 and 2.  We will show by induction on $n$ that $\Gamma$
generates $\Alt(A^n)$.  The hypothesis hands us the base of induction $n=3$
when $|A| = 2$ and $n=2$ when $|A| \ge 3$.  So, we assume a larger value
of $n$ and we assume by induction that the case $n-1$ is already proven.

We consider the two subgroups in $\Alt(A^n)$ that are given by
$\Gamma$-circuits that act respectively on the left $n-1$ symbols or
the right $n-1$ symbols.  By induction, both subgroups are isomorphic to
$\Alt(A^{n-1})$, and we call them $\Alt(A^{n-1})_L$ $\Alt(A^{n-1})_R$.  They
in turn have subgroups $\Alt(A^{n-2})^{|A|}_L$ and $\Alt(A^{n-2})^{|A|}_R$
which are each isomorphic to $\Alt(A^{n-2})^{|A|}$ and each act on the
middle $n-2$ symbols; but in one case the choice of permutation $\alpha
\in \Alt(A^{n-2})$ depends on the leftmost symbol, while in the other case
it depends on the rightmost symbol.  By taking commutators between these
two subgroups, we obtain all permutations in $\Alt(\{a\} \times A^{n-2}
\times \{b\})$ for every pair of symbols $(a,b)$.  Moreover, we can repeat
this construction for every subset of $n-2$ symbols.  Since $n \ge 3$,
and since $n \ge 4$ when $|A| = 2$, we know that $|A^{n-2}| \ge 3$.
We can thus apply \Lem{l:alt} in the next section to the alternating
subgroups that we have obtained.
\end{proof}

\Lem{l:rev} motivates the definition of a canonical reversible gate set
$\Gamma$ for each alphabet $A$.  (Canonical in the sense that it is both
universal and constructed intrinsically from the finite set $A$.)  If $|A|$
is odd, then we let $\Gamma = \Alt(A^2)$.  If $|A| \ge 4$ is even, then
we let $\Gamma = \Sym(A^2)$.  Finally, if $|A| = 2$, then we let $\Gamma
= \Sym(A^3)$.  By \Lem{l:rev}, each of these gate sets is universal.
Moreover, each of these gate sets can be generated by any universal gate
set, possibly with the aid of an ancilla in the even case.

In one version of reversible circuit satisfiability, we choose two subsets
$I,F \subseteq A$, interpreted as initialization and finalization alphabets.
We consider the problem $\RSAT_{A,I,F}(x)$ in which $x$ represents a
reversible circuit $C_x$ of some width $n$ with some universal gate set
$\Gamma$, and a witness $y$ is an input $y \in I^n$ such that $C_x(y)
\in F^n$.

\begin{theorem} Consider $A$, $I$, $F$, and $\Gamma$ with $\Gamma$ universal
and $2 \le |I|, |F| < |A|$.  Then $\RSAT_{A,I,F}$ is Karp $\NP$-hard
and $\shRSAT_{A,I,F}$ is parsimoniously $\shP$-hard.
\label{th:rsat} \end{theorem}

\Thm{th:rsat} is also a standard result in reversible circuit theory, but
again we give a proof because we did not find one.  (Note that if either
$I$ or $F$ only has one element or is all of $A$, then $\RSAT_{A,I,F}$
is trivial.)

\begin{proof} We consider a sequence $\RSAT_i$ of versions of the reversible
circuit problem.  We describe the satisfiability version for each one, and
implicitly define the counting version $\shRSAT_i$ using the same predicate.

\begin{itemize}
\item $\RSAT_1$ uses the binary alphabet $A = \Z/2$ and does not have 
$I$ or $F$.  Instead, some of the input bits are set to $0$ while others
are variable, and the decision output of a circuit is simply the value
of the first bit.

\item $\RSAT_2$ also has $A = \Z/2$ with an even number of input and
output bits.  Half of the input bits and output bits are set to $0$,
while the others are variable.  A circuit $C$ is satisfied by finding an
input/output pair $x$ and $C(x)$ that satisfy the constraints.
 
\item $\RSAT_3$ is $\RSAT_{A,I,F}$ with $I$ and $F$ disjoint and $|A
\setminus (I \cup F)| \ge 2$.

\item $\RSAT_4$ is $\RSAT_{A,I,F}$ with the stated hypotheses of the theorem.
\end{itemize}
We claim parsimonious reductions from $\shCSAT$ to $\shRSAT_1$,
and from $\shRSAT_i$ to $\shRSAT_{i+1}$ for each $i$.

Step 1: We can reduce $\CSAT$ to $\RSAT_1$ through the method of \emph{gate
dilation} and ancillas.  Here an \emph{ancilla} is any fixed input to
the circuit that is used for scratch space; the definition of $\RSAT_1$
includes ancillas.  To define gate dilation, we can let $A$ be any alphabet
with the structure of an abelian group.  If $\alpha:A^k \to A$ is a gate,
then we can replace it with the reversible gate
\[ \beta:A^{k+1} \to A^{k+1} \qquad \beta(x,a) = (x,\alpha(x)+a), \]
where $x \in A^k$ is the input to $\alpha$ and $a \in A$ is an ancilla
which is set to $a=0$ when $\beta$ replaces $\alpha$.  The gate $\beta$
is called a \emph{reversible dilation} of $\alpha$.  We can similarly
replace every irreversible $\COPY$ gate with the reversible gate
\[ \COPY:A^2 \to A^2 \qquad \COPY(x,a) = (x,x+a), \]
where again $a$ is an ancilla set to $a=0$.  Dilations also leave extra
output symbols, but under the satisfiability rule of $\RSAT_1$, we
can ignore them.

In the Boolean case $A = \Z/2$, the reversible $\COPY$ gate is denoted
$\CNOT$ (controlled $\NOT$), while the dilation of $\AND$ is denoted $\CCNOT$
(doubly controlled $\NOT$) and is called the Toffoli gate.  We can add to
this the uncontrolled $\NOT$ gate
\[ \NOT(x) = x+1. \]
These three gates are clearly enough to dilate irreversible Boolean circuits.
(They are also a universal gate set for reversible computation.)

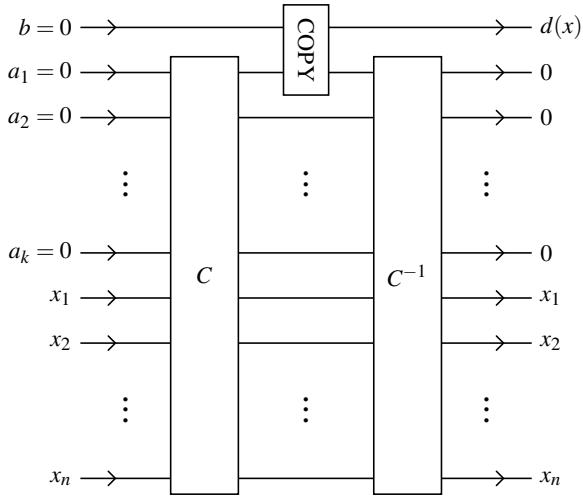
\begin{figure}[htb] \begin{center}
\begin{tikzpicture}[scale=.6,semithick,decoration={markings,
    mark=at position 0.08 with {\arrow{angle 90}},
    mark=at position 0.94 with {\arrow{angle 90}}}]
\draw (0,10) node[anchor=east] {$b=0$};
\draw[postaction={decorate}] (0,10) -- (10,10);
\draw (10,10) node[anchor=west] {$d(x)$};
\draw (0,9) node[anchor=east] {$a_1=0$};
\draw[postaction={decorate}] (0,9) -- (10,9);
\draw (10,9) node[anchor=west] {$0$};
\draw (0,8) node[anchor=east] {$a_2=0$};
\draw[postaction={decorate}] (0,8) -- (10,8);
\draw (10,8) node[anchor=west] {$0$};
\draw (0,5) node[anchor=east] {$a_k=0$};
\draw[postaction={decorate}] (0,5) -- (10,5);
\draw (10,5) node[anchor=west] {$0$};
\draw (0,4) node[anchor=east] {$x_1$};
\draw[postaction={decorate}] (0,4) -- (10,4);
\draw (10,4) node[anchor=west] {$x_1$};
\draw (0,3) node[anchor=east] {$x_2$};
\draw[postaction={decorate}] (0,3) -- (10,3);
\draw (10,3) node[anchor=west] {$x_2$};
\draw (0,0) node[anchor=east] {$x_n$};
\draw[postaction={decorate}] (0,0) -- (10,0);
\draw (10,0) node[anchor=west] {$x_n$};
\draw (1,6.7) node {$\Large \vdots$};
\draw (1,1.7) node {$\Large \vdots$};
\draw[fill=white] (2,-.35) rectangle (3.5,9.35);
\draw (2.75,4.5) node {$C$};
\draw[fill=white] (4.5,8.5) rectangle (5.5,10.5);
\draw (5,9.5) node[rotate=270] {$\COPY$};
\draw (5,6.7) node {$\Large \vdots$};
\draw (5,1.7) node {$\Large \vdots$};
\draw[fill=white] (6.5,-.35) rectangle (8,9.35);
\draw (7.25,4.5) node {$C^{-1}$};
\draw (9,6.7) node {$\Large \vdots$};
\draw (9,1.7) node {$\Large \vdots$};
\end{tikzpicture}
\end{center}
\caption{Using uncomputation to reset ancilla values.}
\label{f:uncomp} \end{figure}

Step 2: We can reduce $\RSAT_1$ to $\RSAT_2$ using the method of
uncomputation.  Suppose that a circuit $C$ in the $\RSAT_1$ problem has
an $n$-bit variable input register $x = (x_1,x_2,\dots,x_n)$ and a $k$-bit
ancilla register $a = (a_1,a_2,\dots,a_k)$.  Suppose that $C$ calculates
decision output $d(x)$ in the $a_1$ position (when $a_1 = 0$ since it is
an ancilla).   Then we can make a new circuit $C_1$ with the same $x$
and $a$ and one additional ancilla bit $b$, defined by applying $C$,
then copying the output to $b$ and negating $b$, then applying $C^{-1}$,
as in \Fig{f:uncomp}.  If $n = k+1$, then $C_1$ is a reduction of $C$ from
$\RSAT_1$ to $\RSAT_2$.  If $n > k+1$, then we can pad $C_1$ with $n-k-1$
more ancillas and do nothing with them to produce a padded circuit $C_2$.
If $n < k+1$, then we can pad $C_1$ with $k+1-n$ junk input bits, and at
the end of $C_1$, copy of these junk inputs to $k+1-n$ of the first $k$
ancillas; again to produce $C_2$.  (Note that $k+1-n \le k$ since we can
assume that $C$ has at least one variable input bit.)  In either of these
cases, $C_2$ is a reduction of $C$ from $\RSAT_1$ to $\RSAT_2$.

Step 3: We can reduce $\RSAT_2$ to $\RSAT_3$ by grouping symbols and
permuting alphabets.   As a first step, let $A_1 = \Z/2 \times \Z/2$
with $I_1 = F_1 = \{(0,0),(1,0)\}$.  Then we can reduce $\RSAT_2$ to
$\RSAT_{A_1,I_1,F_1}$ by pairing each of input or output bit with an
ancilla; we can express each ternary gate over $\Z/2$ in terms of binary
gates over $A_1$.  Now let $A_2$ be any alphabet with disjoint $I_2$ and
$F_2$, and with at least two symbols not in $I_2$ or $F_2$.  Then we can
embed $(A_1,I_1,F_1)$ into $(A_2,I_2,F_2)$ arbitrarily, and extend any gate
$\alpha:A_1^k \to A_1^k$ (with $k \in \{1,2\}$, say) arbitrarily to a gate
$\beta:A_2^k \to A_2^k$ which is specifically an even permutation.  This
reduces $\RSAT_2 = \RSAT_{A_1,I_1,F_1}$ to $\RSAT_3 = \RSAT_{A_2,I_2,F_2}$.

Step 4: Finally, $(A_3,I_3,F_3)$ is an alphabet that is not of our choosing,
and we wish to reduce $\RSAT_3 = \RSAT_{A_2,I_2,F_2}$ to $\RSAT_4 =
\RSAT_{A_3,I_3,F_3}$.  We choose $k$ such that
\[ |A_3|^k \ge |I_3|^k + |F_3|^k + 2. \]
We then let $A_2 = A_3^k$ and $I_2 = I_3^k$, and we choose $F_2 \subseteq A_2
\setminus I_2$ with $|F_2| = |F_3|^k$.  A circuit in $\RSAT_{A_2,I_2,F_2}$
can now be reduced to a circuit in $\RSAT_{A_3,I_3,F_3}$ by grouping
together $k$ symbols in $A_3$ to make a symbol in $A_2$.   Since $I_2 =
I_3^k$, the initialization is the same.  At the end of the circuit, we
convert finalization in $F_2$ to finalization in $F_3^k$ with some unary
permutation of the symbols in $A_2$.
\end{proof}

\subsection{Standard algorithms}
\label{ss:standard}

In this section we will review three standard algorithms that in one way or
another serve as converses to \Thm{th:main}.  Instead of hardness results,
they are all easiness results. (Note that \Thm{th:inshp} produces a
conditional type of easiness, namely predicates that can be evaluated in
polynomial time.)

\begin{theorem} The homology $H_*(X)$ of a finite simplicial complex $X$
can be computed in polynomial time.
\label{th:homology} \end{theorem}

Briefly, \Thm{th:homology} reduces to computing Smith normal forms
of integer matrices \cite{DC:homology}.

\begin{theorem} If $G$ is a fixed finite group and $X$ is a finite,
connected simplicial complex regarded as the computational input, then
$\#H(X,G)$ and $\#Q(X,G)$ are both in $\shP$.
\label{th:inshp} \end{theorem}

\begin{proof} By choosing a spanning tree for the 1-skeleton of $X$, we
can convert its 2-skeleton to a finite presentation $P$ of $\pi_1(X)$.
Then we can describe a homomorphism $f:\pi_1(X) \to G$ by the list of
its values on the generators in $P$.  This serves as a certificate; the
verifier should then check whether the values satisfy the relations in $P$.
This shows that $\#H(X,G)$ is in $\shP$.

The case of $\#Q(X,G)$ is similar but slightly more complicated.  The map $f$
is surjective if and only if its values on the generators in $P$ generate
$G$; the verifier can check this.  The verifier can also calculate the
$\Aut(G)$-orbit of $f$.   Given an ordering of the generators and an
ordering of the elements of $G$, the verifier can accept $f$ only when
it is alphabetically first in its orbit.  Since only surjections are
counted and each orbit is only counted once, we obtain that $\#Q(X,G)$
certificates are accepted.
\end{proof}

In the input to the third algorithm, we decorate a finite simplicial
complex $X$ with a complete ordering of its simplices (of all dimensions)
that refines the partial ordering of simplices given by inclusion.  If there
are $n$ simplices total, then for each $0 \le k \le n$, we let $X_k$ be
subcomplex formed by the first $k$ simplices, so that $X_0 = \emptyset$
and $X_n = X$.  Each $X_k$ has a relative boundary $\bd(X_k)$ in $X$.
(Here we mean boundary in the set of general topology rather than manifold
theory, \ie, closure minus interior.)  We define the \emph{width} of $X$
with its ordering to be the maximum number of simplices in any $\bd(X_k)$.

\begin{theorem} If $G$ is a fixed finite group and $X$ is a finite, connected
simplicial complex with a bounded-width ordering, then $\#H(X,G)$ and
$\#Q(X,G)$ can be computed in polynomial time (non-uniformly in the wdith).
\label{th:width} \end{theorem}

It is easy to make triangulations for all closed surfaces with uniformly
bounded width.  For instance, we can make such a triangulation of an
orientable surface $\Sigma_g$ from a Morse function chosen so that each
regular level is either one or two circles.  With more effort, we can make
a bounded-width triangulation of a Seifert-fibered 3-manifold $M$ using a
bounded-width triangulation of its orbifold base.   Thus \Thm{th:width}
generalizes the formulas of Mednykh \cite{Mednykh:compact} and Chen
\cite{Chen:seifert} in principle, although in practice their formulas
are more explicit and use better decompositions than triangulations.
\Thm{th:width} also applies to 3-manifolds with bounded Heegaard genus,
or more generally bounded Morse width.

\begin{proof} We can calculate $|H(X,G)|$ using the formalism of non-abelian
simplicial cohomology theory with coefficients in $G$ \cite{Olum:nonabelian}.
In this theory, we orient the edges of $X$ and we mark a vertex $x_0 \in
X$ as a basepoint.  A 1-cocycle is then a function from the edges of $X$ to
$G$ that satisfies a natural coboundary condition on each triangle, while
a 0-cochain is a function from the vertices to $G$ that takes the value 1
at $x_0$.  The 1-cocycle set $Z^1(X;G)$ has no natural group structure when
$G$ is non-commutative, while the relative 0-cochain set $C^0(X,x_0;G)$
is a group that acts freely on $Z^1(X;G)$. Then the set of orbits
\[ H^1(X,x_0;G) \defeq Z^1(X;G) / C^0(X,x_0;G). \]
can be identified with the representation set $H(X,G)$, while if $X$ has $v$
vertices, then $C^0(X,x_0;G) \cong G^{v_1}$.  Thus
\[ |H(X,G)| = |Z^1(X;G)|/|G|^{v-1}. \]
Our approach is to compute $|Z^1(X;G)|$ and divide.  We can then also
obtain $|Q(X,G)|$ from $|H(X,G)|$ by applying M\"obius inversion to
equation \eqref{e:homsum}.

The algorithm is an example of \emph{dynamical programming} in computer
science.  Working by induction for each $k$ from $0$ to $n$, it maintains a
vector $v_k$ of non-negative integers that consists of the number of ways to
extend each 1-cocycle on $\bd(X_k)$ to a 1-cocycle of $X_K$.  The dimension
of $v_k$ may be exponential in the number of edges of $\bd(X_k)$, but
since that is bounded, the dimension of $v_k$ is also bounded.  It is
straightforward to compute $v_{k+1}$ from $v_k$ when we pass from $X_k$
to $X_{k+1}$.  If $X_{k+1} \setminus X_k$ is an edge, then $v_{k+1}$
consists of $|G|$ copies of $v_k$.  If $X_{k+1} \setminus X_k$ is a triangle
and $\bd(X_{k+1})$ has the same edges as $\bd(X_k)$, then $v_{k+1}$ is a
subvector of $v_k$.   If $\bd(X_{k+1})$ has fewer edges than $\bd(X_k)$,
then $v_{k+1}$ is obtained from $v_k$ by taking local sums of entries.
\end{proof}

\section{Group theory}

In this section we collect some group theory results.  We do not consider any
of these results to be especially new, although we found it challenging
to prove \Thm{th:rubik}.

\subsection{Generating alternating groups}
\label{ss:alt}

\begin{lemma}[\Cf\ {\cite[Lem. 7]{DGK:shuffles}}] Let $S$ be a finite set
and let $T_1,T_2,\dots,T_n \subseteq S$ be a collection of subsets with at
least 3 elements each, whose union is $S$, and that form a connected graph
under pairwise intersection.  Then the permutation groups $\Alt(T_i)$
together generate $\Alt(S)$.
\label{l:alt} \end{lemma}

\begin{proof} We argue by induction on $|S \setminus T_1|$.  If $T_1 =
S$, then there is nothing to prove.  Otherwise, we can assume (possibly
after renumbering the sets) that there is an element $a \in T_1 \cap T_2$
and an element $b \in T_2 \setminus T_1$.  Let $\alpha \in \Alt(T_2)$ be a
3-cycle such that $\alpha(a) = b$.  Then the 3-cycles in $\Alt(T_1)$, and
their conjugates by $\alpha$, and $\alpha$ itself if it lies in $\Alt(T_1
\cup \{b\})$, include all 3-cycles in $\Alt(T_1 \cup \{b\})$.  Thus we
generate $\Alt(T_1 \cup \{b\})$ and we can replace $T_1$ by $T_1 \cup \{b\}$.
\end{proof}

\subsection{Joint surjectivity}
\label{ss:joint}

Recall the existence half of the Chinese remainder theorem:  If 
$d_1,d_2,\ldots,d_n$ are pairwise relatively prime integers, then the
canonical homomorphism
\[ f:\Z \to \Z/d_1 \times \Z/d_2 \times \dots \times \Z/d_n \]
is (jointly) surjective.   The main hypothesis is ``local" in the
sense that it is a condition on each pair of divisors $d_i$ and $d_j$,
namely $\mathrm{gcd}(d_i,d_j) = 1$.  For various purposes, we will need
non-commutative joint surjectivity results that resemble the classic Chinese
remainder theorem.  (But we will not strictly generalize the Chinese
remainder theorem, although such generalizations exist.)  Each version
assumes a group homomorphism
\[ f:K \to G_1 \times G_2 \times \dots \times G_n \]
that surjects onto each factor $G_i$, and assumes certain other local
hypotheses, and concludes that $f$ is jointly surjective.  Dunfield-Thurston
\cite[Lem. 3.7]{DT:random} and the first author \cite[Lem. 3.5]{K:zdense}
both have results of this type and call them ``Hall's lemma", but Hall
\cite[Sec. 1.6]{Hall:eulerian} only stated without proof a special case
of Dunfield and Thurston's lemma.  Ribet \cite[Lem. 3.3]{Ribet:adic} also
has such a lemma with the proof there attributed to Serre.  In this paper,
we will start with a generalization of Ribet's lemma.

We define a group homomorphism
\[ f:K \to G_1 \times G_2 \times \dots \times G_n \]
to be \emph{$k$-locally surjective} for some integer $1 \le k \le n$ if it
surjects onto every direct product of $k$ factors.  Recall also that if $G$
is a group, then $G' = [G,G]$ is a notation for its commutator subgroup.

\begin{lemma}[After Ribet-Serre {\cite[Lem. 3.3]{Ribet:adic}}] 
Let 
\[ f:K \to G_1 \times G_2 \times \dots \times G_n \]
be a 2-locally surjective group homomorphism, such that also
its abelianization
\[ f_\ab:K \to (G_1)_\ab \times (G_2)_\ab \times \dots \times (G_n)_\ab \]
is $\ceil{(n+1)/2}$-locally surjective.  Then
\[ f(K) \ge G_1' \times G_2' \times \dots \times G_n'. \]
\eatline \label{l:ribet} \end{lemma}

\begin{proof} We argue by induction on $n$.  If $n=2$, then there is
nothing to do.  Otherwise let $t = \ceil{(n+1)/2}$ and note
that $n > t > n/2$.  Let 
\[ \pi:G_1 \times G_2 \times \dots \times G_n \to
    G_1 \times G_2 \times \dots \times G_t \]
be the projection onto the first $t$ factors.  Then $\pi \circ f$ satisfies
the hypotheses, so
\[ \pi(f(K)) \ge G_1' \times G_2' \times \dots \times G_t'. \]
Morever, $(\pi \circ f)_\ab$ is still $t$-locally surjective, which is to
say that
\[ \pi(f(K))_\ab = (G_1)_\ab \times (G_2)_\ab \times
    \dots \times (G_t)_\ab. \]
Putting these two facts together, we obtain
\[ \pi(f(K)) = G_1 \times G_2 \times \dots \times G_t. \]
Repeating this for any $t$ factors, we conclude that $f$ is $t$-locally
surjective.

Given any two elements $g_t,h_t \in G_t$, we can use $t$-local surjectivity
to find two elements
\begin{multline*} (g_1,g_2,\dots,g_{t-1},g_t,1,1,\dots,1), \\
    (1,1,\dots,1,h_t,h_{t+1},\dots,h_n) \in f(K).
\end{multline*}
Their commutator then is $[g_t,h_t] \in G_t \cap f(K)$.   Since $g_t$
and $h_t$ are arbitrary, we thus learn that $G'_t \le f(K)$, and since
this construction can be repeated for any factor, we learn that
\[ f(K) \ge G_1' \times G_2' \times \dots \times G_n', \]
as desired.
\end{proof}

We will also use a complementary result, Goursat's lemma, which can be used
to establish 2-local surjectivity.  (Indeed, it is traditional in some papers
to describe joint surjectivity results as applications of Goursat's lemma.)

\begin{lemma}[Goursat \cite{Goursat:divisions,BSZ:goursat}] Let $G_1$ and
$G_2$ be groups and let $H \leq G_1 \times G_2$ be a subgroup that
surjects onto each factor $G_i$.  Then there exist normal subgroups $N_i
\normaleq G_i$ such that $N_1 \times N_2 \le H$ and $H/(N_1 \times N_2)$
is the graph of an isomorphism $G_1/N_1 \cong G_2/N_2$.
\label{l:goursat} \end{lemma}

For instance, if $G_1$ is a simple group, then either $H = G_1 \times G_2$
or $H$ is the graph of an isomorphism $G_1 \cong G_2$.  In other words,
given a joint homomorphism
\[ f = f_1 \times f_2:K \to G_1 \times G_2 \]
which surjects onto each factor, either $f$ is surjective or $f_1$ and $f_2$
are equivalent by an isomorphism $G_1 \cong G_2$.  We can combine this
with the perfect special case of \Lem{l:ribet} to obtain exactly Dunfield
and Thurston's version.

\begin{lemma}[{\cite[Lem. 3.7]{DT:random}}] If 
\[ f:K \to G_1 \times G_2 \times \dots \times G_n \]
is a group homomorphism to a direct product of non-abelian simple groups,
and if no two factor homomorphism $f_i:K \to G_i$ and $f_j:K \to G_j$
are equivalent by an isomorphism $G_i \cong G_j$, then $f$ is surjective.
\label{l:sjoint}\end{lemma}

\begin{corollary} Let $K$ be a group and let
\[ N_1,N_2,\dots,N_n \normal K \]
be distinct maximal normal subgroups with non-abelian simple quotients
$G_i = K/N_i$.  Then
\[ G_1 \cong (N_2 \cap N_3 \cap \dots \cap N_n)/
    (N_1 \cap N_2 \cap \dots \cap N_n). \]
\label{c:sjoint} \end{corollary}

\begin{proof} We can take the product of the quotient maps to obtain
a homomorphism
\[ f:K \to G_1 \times G_2 \times \dots \times G_n \]
that satisfies \Lem{l:sjoint}.  Thus we can restrict $f$ to
\[ f^{-1}(G_1) = N_2 \cap N_3 \cap \dots \cap N_n \]
to obtain a surjection
\[ f:N_2 \cap N_3 \cap \dots \cap N_n \onto G_1. \]
This surjection yields the desired isomorphism.
\end{proof}

We will use a more direct corollary of \Lem{l:goursat}.  We say that
a group $G$ is \emph{normally Zornian} if every normal subgroup of $G$
is contained in a maximal normal subgroup.  Clearly every finite group is
normally Zornian, and so is every simple group.  A more interesting result
implied by Neumann \cite[Thm. 5]{Neumann:remarks} is that every finitely
generated group is normally Zornian.  (Neumann's stated result is that
every subgroup is contained in a maximal subgroup, but the proof works
just as well for normal subgroups.  He also avoided the axiom of choice
for this result, despite our reference to Zorn's lemma.)  Recall also the
standard concept that a group $H$ is \emph{involved} in another group $G$
if $H$ is a quotient of a subgroup of $G$.

\begin{lemma} Suppose that 
\[ f:K \to G_1 \times G_2 \]
is a group homomorphism that surjects onto the first factor $G_1$, and
that $G_1$ is normally Zornian.  Then:
\begin{enumerate}
\item If no simple quotient of $G_1$ is involved in $G_2$, then 
$f(K)$ contains $G_1$.
\item If $f$ surjects onto $G_2$, and no simple quotient of $G_1$
is a quotient of $G_2$, then $f$ is surjective.
\end{enumerate}
\label{l:zorn} \end{lemma}

\begin{proof} Case 1 reduces to case 2, since we can replace $G_2$ by the
projection of $f(K)$ in $G_2$.  In case 2, \Lem{l:goursat} yields isomorphic
quotients $G_1/N_1 \cong G_2/N_2$.  Since $G_1$ is normally Zornian, we
may further quotient $G_1/N_1$ to produce a simple quotient $Q$, and we
can quotient $G_2/N_2$ correspondingly.
\end{proof}

Finally, we have a lemma to calculate the simple quotients of a direct
product of groups.

\begin{lemma} If
\[ f:G_1 \times G_2 \times \cdots \times G_n \onto Q \]
is a group homomorphism from a direct product  to a non-abelian simple
quotient, then it factors through a quotient map $f_i:G_i \to Q$ for a
single value of $i$.
\label{l:prodquo} \end{lemma}

\begin{proof} The lemma clearly reduces to the case $n=2$ by induction.  If
\[ f:G_1 \times G_2 \onto Q \]
is a simple quotient, then $f(G_1)$ and $f(G_2)$ commute with each other,
so they are normal subgroups of the group that they generate, which by
hypothesis is $Q$.  So each of $f(G_1)$ and $f(G_2)$ is either trivial
or equals $Q$.  Since $Q$ is non-commutative, then $f(G_1)$ and $f(G_2)$
cannot both be $Q$, again because they commute with each other.  Thus one
of $G_1$ and $G_2$ is in the kernel of $f$, and $f$ factors through a
quotient of the other one.
\end{proof}

\subsection{Integer symplectic groups}

Recall that for any integer $g \ge 1$ and any commutative ring $A$,
there is an integer symplectic group $\Sp(2g,A)$, by definition the set of
automorphisms of the free $A$-module $A^{2g}$ that preserves a symplectic
inner product.  Likewise the projective symplectic group $\PSp(2g,A)$
is the quotient of $\Sp(2g,A)$ by its center (which is trivial in
characteristic 2 and consists of $\pm I$ otherwise).  For each prime
$p$ and each $g \ge 1$, the group $\PSp(2g,\Z/p)$ is a finite simple
group, except for $\PSp(2,\Z/2)$, $\PSp(2,\Z/3)$, and $\PSp(4,\Z/2)$
\cite[Thm. 11.1.2]{Carter:simple}.  Moreover, $\PSp(2g,\Z/p)$ is never
isomorphic to an alternating group when $g \ge 2$ (because it has the
wrong cardinality).

We want to apply \Lem{l:zorn} to the symplectic group $\Sp(2g,\Z)$,
since it is the quotient of the mapping class group $\MCG_*(\Sigma_g)$
by the Torelli group $\Tor_*(\Sigma_g)$.  To this end, we can
describe its simple quotients when $g \ge 3$.

\begin{lemma} If $g \ge 3$, then the simple quotients of $\Sp(2g,\Z)$
are all of the form $\PSp(2g,\Z/p)$, where $p$ is prime and the quotient
map is induced by the ring homomorphism from $\Z$ to $\Z/p$.
\label{l:symplectic} \end{lemma}

As the proof will indicate, \Lem{l:symplectic} is a mutual corollary of
two important results due to others:  the congruence subgroup property of
Mennicke and Bass-Lazard-Serre, and the Margulis normal subgroup theorem.

Note that the finite simple quotients of $\Sp(4,\Z)$ are only slightly
different.  The best way to repair the result in this case is to replace
both $\Sp(4,\Z)$ and $\Sp(4,\Z/2)$ by their commutator subgroups of index 2.
Meanwhile given the well-known fact that $\PSp(2,\Z) \cong C_2 * C_3$,
any simple group generated by an involution and an element of order 3 is
a simple quotient of $\Sp(2,\Z)$, and this is a very weak restriction.
However, we only need \Lem{l:symplectic} for large $g$.

\begin{proof} We note first that $\Sp(2g,\Z)$ is a perfect group when
$g \ge 3$, so every possible simple quotient is non-abelian, and every
such quotient is also a quotient of $\PSp(2g,\Z)$.  It is a special
case of the Margulis normal subgroup theorem \cite{Margulis:discrete}
that $\PSp(2g,\Z)$ is \emph{just infinite} for $g \ge 2$, meaning that
all quotient groups are finite.   Meanwhile, a theorem of Mennicke and
Bass-Lazard-Serre \cite{Mennicke:zur,BLS:fini} says that $\Sp(2g,\Z)$ has
the \emph{congruence subgroup property}, meaning that all finite quotients
factor through $\Sp(2g,\Z/n)$ for some integer $n > 1$.  Every finite
quotient of $\PSp(2g,\Z)$ likewise factors through $\PSp(2g,\Z/n)$, so we
only have to find the simple quotients of $\PSp(2g,\Z/n)$.

Clearly if a prime $p$ divides $n$, then $\PSp(2g,\Z/p)$ is a simple
quotient of $\PSp(2g,\Z/n)$.  We claim that there are no others.  Let $N$
be the kernel of the joint homomorphism
\[f:\PSp(2g,\Z/n) \to \prod_{p|n \text{ prime}} \PSp(2g,\Z/p). \]
If $\PSp(2g,\Z/n)$ had another simple quotient, necessarily non-abelian, then
by \Cor{c:sjoint}, it would also be a simple quotient of $N$.  It is easy
to check that $N$ is nilpotent, so all of its simple quotients are abelian.
\end{proof}

\subsection{Rubik groups}
\label{ss:rubik}

Recall the notation that $G' = [G,G]$ is the commutator subgroup of a
group $G$.

If $G$ is a group and $X$ is a $G$-set, then we define the $G$-set
symmetric group $\Sym_G(X)$ to be the group of permutations of $X$ that
commute with the action of $G$.  (Equivalently, $\Sym_G(X)$ is the group
of automorphisms of $X$ as a $G$-set.)  In the case that there are only
finitely many orbits, we define the \emph{Rubik group} $\Rub_G(X)$ to be
the commutator subgroup $\Sym_G(X)'$.  (For instance, the actual Rubik's
Cube group has a subgroup of index two of the form $\Rub_G(X)$, where $G =
C_6$ acts on a set $X$ with 12 orbits of order 2 and 8 orbits of order 3.)

If every $G$-orbit of $X$ is free and $X/G$ has $n$ elements, then we can
recognize $\Sym_G(X)$ as the restricted wreath product
\[ \Sym_G(X) \cong G \wr_{X/G} \Sym(X/G) \cong G \wr_n \Sym(n). \]
We introduce the more explicit notation
\begin{align*}
\Sym(n,G) &\defeq G \wr_n \Sym(n) \\
\Alt(n,G) &\defeq G \wr_n \Alt(n) \\
\Rub(n,G) &\defeq \Sym(n,G)'.
\end{align*}
We can describe $\Rub(n,G)$ as follows.  Let $G_\ab$ be the abelianization
of $G$, and define a map $\sigma:G^n \to G_\ab$ by first abelianizing $G^n$
and then multiplying the $n$ components in any order. Let $\AD(n,G) \le G^n$
($\AD$ as in ``anti-diagonal") be the kernel of $\sigma$.  Then:

\begin{proposition} For any integer $n > 1$ and any group $G$, the commutator
subgroup of $\Sym(n,G)$ is given by
\[ \Rub(n,G) = \AD(n,G) \rtimes \Alt(n). \]
\end{proposition}

\begin{proof} It is easy to check that $(\ker \sigma) \rtimes \Alt(n)$
is a normal subgroup of $\Sym(n,G)$ and that the quotient is the abelian
group $G_\ab \times C_2$. This shows that
\[ \AD(n,G) \rtimes \Alt(n) \supseteq \Rub(n,G). \]

To check the opposite inclusion, note that $\AD(n,G) \rtimes \Alt(n)$
is generated by the union of $(G')^n$, $\Alt(n) = \Sym(n)'$, and all
permutations of elements of the form
\[ (g,g^{-1},1,\dots,1) \in G^n.\]
Clearly $\Rub(n,G)$ contains the former two subsets.  Since
\[ (g,g^{-1},1,\dots,1) = [(g,1,1,\dots,1),(1 \ 2)]\]
(and similarly for other permutations),  we see
\[ \AD(n,G) \rtimes \Alt(n) \le \Rub(n,G).\]
We conclude with the desired equality.
\end{proof}

The main result of this section is a condition on a group homomorphism
to $\Rub(n,G)$ that guarantees that it is surjective.  We say that a group
homomorphism
\[ f:K \to \Sym(n,G)\]
is \emph{$G$-set $k$-transitive} if it acts transitively on ordered lists
of $k$ elements that all lie in distinct $G$-orbits.

\begin{theorem} Let $G$ be a group and let $n \ge 7$ be an integer such
that $\Alt(n-2)$ is not a quotient of $G$.  Suppose that a homomorphism
\[ f:K \to \Rub(n,G) \]
is $G$-set 2-transitive and that its composition with the projection
$\Rub(n,G) \to \Alt(n)$ is surjective.  Then $f$ is surjective.
\label{th:rubik} \end{theorem}

\begin{proof}[Proof of \Thm{th:rubik}] In the proof we will mix Cartesian
product notation for elements of $G^n$ with cycle notation for permutations.
The proof is divided into three steps.

Step 1: We let $H = f(K)$, and we consider its normal subgroup
\[ D \defeq H \cap G^n. \]
We claim that $D$ is $2$-locally surjective.  To this end, we look at the
subgroup $\Alt(n-2) \le \Alt(n)$ that fixes the last two letters (say).
Then there is a projection
\[ \pi:G^n \rtimes \Alt(n-2) \to G^2 \times \Alt(n-2) \]
given by retaining only the last two coordinates of $g \in G^n$.  We let
\[ J = \pi(H \cap (G^n \rtimes \Alt(n-2))). \]
Since $H$ is $G$-set 2-transitive, the group $J$ surjects onto $G^2$;
since $H$ surjects onto $\Alt(n)$, $J$ surjects onto $\Alt(n-2)$.  Thus we
can apply \Lem{l:zorn} to the inclusion
\[ J \leq G^2 \times \Alt(n-2). \]
Since $\Alt(n-2)$ is not a quotient of $G$ and therefore not $G^2$ either
(by \Lem{l:prodquo}), we learn that
\[ J = G^2 \times \Alt(n-2) \]
and that 
\[ G^2 \leq H \cap (G^n \rtimes \Alt(n-2)). \]
So the group $D = H \cap G^n$ surjects onto the last two coordinates
of $G^n$.  Since we can repeat the argument for any two coordinates, $D$
is 2-locally surjective.

Step 2: Suppose that $G$ is abelian.  Then $D$ is a subgroup of $G^n$
which is 2-locally surjective.  Since $G^n$ is abelian, conjugation of
elements of $D$ by elements of $H$ that surject onto $\Alt(n)$ coincides
with conjugation by $\Alt(n)$; thus $D$ is $\Alt(n)$-invariant.  By step 1,
for each $g_1 \in G$, there exists an element
\[ d_1 \defeq (g_1,1,g_3,g_4,\dots,g_n) \in D. \]
We now form a commutator with elements in $\Alt(n)$ to obtain
\begin{align*}
d_2 &\defeq [d_1,(1\;2)(3\;4)]
    = (g_1,g_1^{-1},g_3g_4^{-1},g_3^{-1}g_4,1,\dots,1) \in D. \\
d_3 &\defeq [d_2,(1 \ 2 \ 5)(3 \ 4)(6 \ 7)]
    = (g,1,1,1,g^{-1},1,\dots,1) \in D.
\end{align*}
The $\Alt(n)$-orbit of $d_3$ generates $\AD(n,G)$, thus $D = \AD(n,G)$.

Step 3: In the general case, step 2 tells us that $D_\ab = \AD(n,G_\ab)$
is $(n-1)$-locally surjective.  This together with step 1 tells us that $D
\le G^n$ satisfies the hypothesis of \Lem{l:ribet}, which tells us that $D
= \AD(n,G)$.  It remains only to show that $\Alt(n) \le H$.  It suffices
to show that $H/D$ contains (indeed is) $\Alt(n)$ in the quotient group
\[ \Alt(n,G)/D \cong (G^n/D) \rtimes \Alt(n)
    \cong (G^n/D) \times \Alt(n). \]
Now let $D_0 = (G^n \cap H)/D$, so that $H$ surjects onto $D_0 \times
\Alt(n)$.  Since $\Alt(n)$ is not a quotient of $D_0$ (for one reason,
because $G^n/D$ is abelian), we can thus apply \Lem{l:zorn} to conclude
that $H/D$ contains $\Alt(n)$.
\end{proof}

\begin{lemma} If $G$ is a group and $n \ge 5$, then $\Rub(n,G)/\AD(n,G)
\cong \Alt(n)$ is the unique simple quotient of $\Rub(n,G)$.
\label{l:rubikquo} \end{lemma}

\begin{proof} We first claim that $\Rub(n,G)$ is a perfect group.  For any
two elements $g,h \in G$, we can take commutators such as
\begin{multline*}
[(g_1,g_1^{-1},1,1,\dots,1), (g_2,1,g_2^{-1},1,\dots,1)] \\
    = ([g_1,g_2],1,1,\dots,1) \in \AD(n,G)',
\end{multline*}
to conclude that
\[ (G')^n = \AD(n,G)' \le \Rub(n,G)'. \]
We can thus quotient $\Rub(n,G)$ by $(G')^n$ and replace $G$ by $G_\ab$,
or equivalently assume that $G$ is abelian.  In this case, we can take
commutators such as
\begin{multline*}
[(g,1,g^{-1},1,1,\dots,1),(1\; 2)(4\; 5)] \\
    = (g,g^{-1},1,1,\dots,1) \in \Rub(n,G)'
\end{multline*}
to conclude that $\AD(n,G) \le \Rub(n,G)'$.  Meanwhile $\Alt(n) \le
\Rub(n,G)'$ because it is a perfect subgroup.  Thus $\Rub(n,G)$ is perfect.

Suppose that
\[ f:\Rub(n,G) \onto Q \]
is a second simple quotient map, necessarily non-abelian.  Then \Cor{c:sjoint}
tells us that $f$ is also surjective when restricted to $\AD(n,G)$.  If $G$
is abelian, then so is $\AD(n,G)$ and this is immediately impossible.
Otherwise we obtain that the restriction of $f$ to $\AD(n,G)' = (G')^n$ is
again surjective, and we can apply \Lem{l:prodquo} to conclude that
$f|_{(G')^n}$ factors through a quotient $h:G' \to Q$ on a single factor.
But then $(\ker f) \cap (G')^n$ would not be invariant under conjugation
by $\Alt(n)$ even though it is the intersection of two normal subgroups
of $\Rub(n,G)$, a contradiction.
\end{proof}

\section{Proof of \Thm{th:main}}
\label{s:proof}

In this section, we will prove \Thm{th:main} in three stages.
In \Sec{ss:zombies}, we define an ad hoc circuit model called $\ZSAT$
in which the alphabet has a group action and also has an unwanted zombie
state $z$.  Despite its contrived features, $\RSAT$ reduces to $\ZSAT$,
which is thus $\shP$-complete.  In \Sec{ss:refine}, we refine \Thm{th:dt}
of Dunfield and Thurston in several ways for our purposes.  Finally in
\Sec{ss:mcg}, we build a homology 3-sphere $M$ from a $\ZSAT$ circuit that
satisfies the requirements of \Thm{th:main}.

\subsection{Zombies}
\label{ss:zombies}

Let $G$ be a non-trivial finite group and let $A$ be an alphabet which
is a $G$-set with a single fixed point $z$, the \emph{zombie symbol},
and otherwise with free orbits.  We choose two $G$-invariant alphabets
$I,F \subsetneq A \setminus \{z\}$, and we assume that
\begin{eq}{e:zineq}
|I|,|F| \ge 2|G| \qquad I \ne F \qquad |A| \ge 2|I \cup F| + 3|G| + 1.
\end{eq}
(The second and third conditions are for convenience rather than necessity.)
With these parameters, we define a planar circuit model that we denote
$\ZSAT_{G,A,I,F}$ that is the same as $\RSAT_{A,I \cup \{z\},F \cup
\{z\}}$ as defined in \Sec{ss:circuits}, \emph{except} that the gate set
is $\Rub_G(A^2)$.  This gate set is not universal in the sense of $\RSAT$
because every gate and thus every circuit is $G$-equivariant.  (One can
show that it is universal for $G$-equivariant circuits by establishing an
analogue of \Lem{l:rev} with the aid of \Thm{th:rubik}, but we will not need
this.)  More explicitly, in the $\ZSAT$ model we consider $G$-equivariant
planar circuits $C$ that are composed of binary gates in $\Rub_G(A^2)$,
and satisfiability is defined by the equation $C(x) = y$ with $x \in
(I \cup \{z\})^n$ and $y \in (F \cup \{z\})^n$.

\begin{lemma} $\shZSAT_{G,A,I,F}$ is almost parsimoniously $\shP$-complete.
More precisely, if $c \in \shP$, then there is an $f \in \FP$ such that
\begin{eq}{e:zk} \shZSAT_{G,A,I,F}(f(x)) = |G|c(x)+1. \end{eq}
\eatline \label{l:zsat} \end{lemma}

Equation \eqref{e:zk} has the same form as equation \eqref{e:hq}, and for
an equivalent reason: The input $(z,z,\dots,z)$ trivially satisfies any
$\ZSAT$ circuit (necessarily at both ends), while $G$ acts freely on the
set of other circuit solutions.

\begin{proof} We take the convention that $A$ is a left $G$-set.  We choose
a subset $A_0 \subsetneq A$ that has one representative from each free
$G$-orbit of $A$.  (In other words, $A_0$ is a section of the free orbits.)

We say that a data state $(a_1,a_2,\dots,a_n)$ of a $\ZSAT$ circuit of
width $n$ is \emph{aligned} if it has no zombie symbols and if there is a
single element $g \in G$ such that $ga_i \in A_0$ for all $i$.   The idea
of the proof is to keep zombie symbols unchanged (which is why they are
called zombies) and preserve alignment in the main reduction, and then add
a postcomputation that converts zombies and misaligned symbols into warning
symbols in a separate warning alphabet.  The postcomputation cannot work
if all symbols are zombies, but it can work in all other cases.

More precisely, we let $W \subseteq A \setminus (I \cup F \cup \{z\})$
be a $G$-invariant subalphabet of size $|I \cup F|+2|G|$ which we call the
\emph{warning alphabet}, and we distinguish two symbols $z_1,z_2 \in W$ in
distinct orbits.  Using \Thm{th:rsat} as a lemma, we will reduce a circuit
$C$ in the planar, reversible circuit model $\RSAT_{(I \cup F)/G,I/G,F/G}$
with binary gates to a circuit $D$ in $\ZSAT_{G,A,I,F}$.  To describe the
reduction, we identify each element of $(I \cup F)/G$ with its lift in $A_0$.

We let $D$ have the same width $n$ as $C$.  To make $D$, we convert each
binary gate $\gamma$ of the circuit $C$ in $\RSAT_{(I \cup F)/G,I/G,F/G}$
to a gate $\delta$ in $\ZSAT_{G,A,I,F}$ in sequence.  After all of these
gates, $D$ will also have a postcomputation stage.  Given $\gamma$,
we define $\delta$ as follows:
\begin{enumerate}
\item Of necessity,
\[ \delta(z,z) = (z,z). \]
\item If $a \in I \cup F$, then
\[ \delta(z,a) = (z,a) \qquad \delta(a,z) = (a,z). \]
\item If $a_1,a_2 \in (I \cup F) \cap A_0$, $g_1,g_2 \in G$, and 
\[ \gamma(a_1,a_2) = (b_1,b_2), \]
then
\[ \delta(g_1a_1,g_2a_2) = (g_1b_1,g_2b_2). \]
\item We extend $\delta$ to the rest of $A^2$ so that $\delta \in
\Rub_G(A^2)$.
\end{enumerate}
By cases 1 and 2, zombie symbols stay unchanged.  Cases 1, 2, and 3 together
keep the computation within the subalphabet $I \cup F \cup \{z\}$, while
case 3 preserves alignments, as well as misalignments.

The postcomputation uses a gate $\alpha:A^2 \to A^2$ such that:
\begin{enumerate}
\item Of necessity,
\[ \alpha(z,z) = (z,z). \]
\item If $a \in I \cap A_0$, then
\[ \alpha(z,a) = (z_1,a) \qquad \alpha(a,z) = (z_2,a) \]
\item If $a_1,a_2 \in I \cup F$ are misaligned, then
\[ \alpha(a_1,a_2) = (\beta(a_1),a_2) \]
for some $G$-equivariant bijection 
\[ \beta: I \cup F \to W \setminus (Gz_1 \cup Gz_2). \]
\item If $a_1,a_2 \in I \cup F$ are aligned, then
\[ \alpha(a_1,a_2) = (a_1,a_2). \]
\item We extend $\alpha$ to the rest of $A^2$ so that $\alpha \in
\Rub_G(A^2)$.
\end{enumerate}
We apply this gate $\alpha$ to each adjacent pair of symbols $(a_i,a_{i+1})$
for $i$ ranging from $1$ to $n-1$ in order.  The final effect is that,
if some (but not all) of the input symbols are zombies, or if any two
symbols are misaligned, then the postcomputation in $D$ creates symbols
in the warning alphabet $W$.

Any input to $D$ with either zombies or misaligned symbols cannot finalize,
since the main computation preserves these syndromes and the postcomputation
then produces warning symbols that do not finalize.  The only spurious input
that finalizes is the all-zombies state $(z,z,\dots,z)$, and otherwise
each input that $C$ accepts yields a single aligned $G$-orbit.  Thus we
obtain the relation
\[ \#D = |G|\#C + 1 \]
between the number of inputs that satisfy $C$ and the number that satisfy
$D$, as desired.
\end{proof}

\subsection{\Thm{th:dt} refined}
\label{ss:refine}

In this subsection and the next one, we let $G$ be a fixed finite simple
group, and we use ``eventually" to mean ``when the genus $g$ is sufficiently
large".

Recall from \Sec{s:intro} that we consider several sets of homomorphisms
of the fundamental group of the surface $\Sigma_g$ to $G$:
\begin{align*}
\hR_g(G) &\defeq \{f:\pi_1(\Sigma_g) \to G\} \\
R_g(G) &\defeq \{f: \pi_1(\Sigma_g) \onto G\} \subseteq \hR_g(G) \\
R^s_g(G) &\defeq \{f \in R_g \mid \sch(f) = s\}.
\end{align*}
For convenience we will write $R_g = R_g(G)$, etc., and only give the
argument of the representation set when the target is some group other
than $G$.

The set $\hR_g$ has an action of $K = \Aut(G)$ and a commuting action of
$\MCG_*(\Sigma_g)$, so we obtain a representation map
\[ \rho:\MCG_*(\Sigma_g) \to \Sym_K(\hR_g). \]
Since $\MCG_*(\Sigma_g)$ is perfect for $g \ge 3$ \cite[Thm. 5.2]{FM:primer}
(and we are excluding small values of $g$), we can let the target be
$\Rub_K(\hR_g)$ instead.  Now $R_g$ and $R_g^0$ are both invariant subsets
under both actions; in particular the representation map projects to maps
to $\Sym_K(\hR_g \setminus R_g)$ and $\Sym_K(R^0_g)$.  At the same time,
$\MCG_*(\Sigma_g)$ acts on $H_1(\Sigma_g) \cong \Z^{2g}$, and we get a
surjective representation map
\[ \tau:\MCG_*(\Sigma_g) \to \Sp(2g,\Z), \]
whose kernel is by definition the Torelli group $\Tor_*(\Sigma_g)$.

The goal of this subsection is the following theorem.

\begin{theorem} The image of the joint homomorphism
\begin{multline*}
\rho_{R^0_g} \times \rho_{\hR_g \setminus R_g} \times \tau:\MCG_*(\Sigma_g) \\
    \to \Rub_K(R^0_g) \times \Rub_K(\hR_g \setminus R_g) \times \Sp(2g,\Z)
\end{multline*}
eventually contains $\Rub_K(R^0_g)$.
\label{th:refine} \end{theorem}

Comparing \Thm{th:refine} to the second part of \Thm{th:dt}, it says that
\Thm{th:dt} still holds for the smaller Torelli group $\Tor_*(\Sigma_g)$,
and after that the action homomorphism is still surjective if we lift
from $\Alt(R^0_g/K)$ to $\Rub_K(R^0_g)$.   Its third implication is
that we can restrict yet further to the subgroup of $\Tor_*(\Sigma_g)$
that acts trivially on $\hR_g \setminus R_g$, the set of non-surjective
homomorphisms to $G$.

The proof uses a lemma on relative sizes of representation sets.

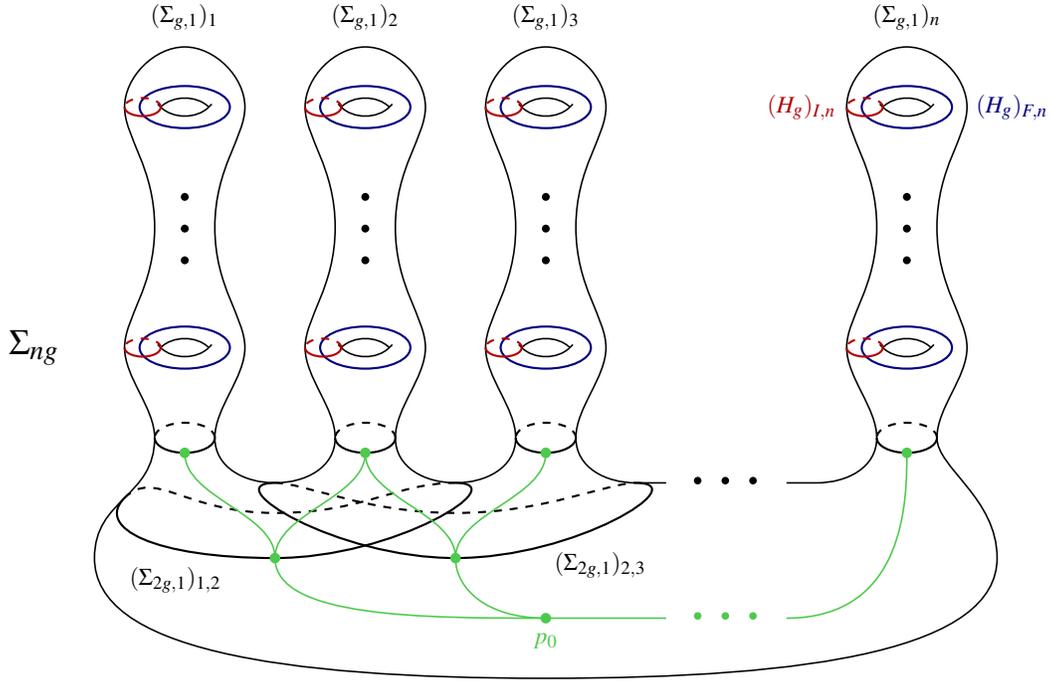
\begin{figure*}[t] \begin{center}
\begin{tikzpicture}[scale=.4,semithick]
\foreach \x in {0,6,12,24} {
\begin{scope}[shift={(\x,0)}]
\draw (1,0) .. controls (1,1) and (0,2) .. (0,3)
    .. controls (0,4) and (1,5) .. (1,7)
    .. controls (1,9) and (0,10) .. (0,11)
    .. controls (0,12) and (1,13) .. (2,13)
    .. controls (3,13) and (4,12) .. (4,11)
    .. controls (4,10) and (3,9) .. (3,7)
    .. controls (3,5) and (4,4) .. (4,3)
    .. controls (4,2) and (3,1) .. (3,0);
\draw (2,7.75) node {\scalebox{3}{$\vdots$}};
\draw (1.2,3) .. controls (1.6,3.4) and (2.4,3.4) .. (2.8,3);
\draw (1.1,3.1) -- (1.2,3) .. controls (1.6,2.6) and (2.4,2.6) .. (2.8,3)
    -- (2.9,3.1);
\draw[darkblue,thick] (2,3) ellipse (1.5 and .7);
\draw[darkred,thick,dashed] (1.2,3) arc (0:180:.6 and .3);
\draw[darkred,thick] (0,3) arc (-180:0:.6 and .3);
\draw (1.2,11) .. controls (1.6,11.4) and (2.4,11.4) .. (2.8,11);
\draw (1.1,11.1) -- (1.2,11) .. controls (1.6,10.6) and (2.4,10.6) .. (2.8,11)
    -- (2.9,11.1);
\draw[darkblue,thick] (2,11) ellipse (1.5 and .7);
\draw[darkred,thick,dashed] (1.2,11) arc (0:180:.6 and .3);
\draw[darkred,thick] (0,11) arc (-180:0:.6 and .3);
\draw[thick,dashed] (3,0) arc (0:180:1 and .5);
\draw[thick] (1,0) arc (-180:0:1 and .5);
\fill[medgreen] (2,-.5) circle (.175);
\end{scope} }
\draw (1,0) .. controls (1,-2) and (-1,-2) .. (-1,-4)
    .. controls (-1,-7) and (5,-8) .. (14,-8)
    .. controls (23,-8) and (29,-7) .. (29,-4)
    .. controls (29,-2) and (27,-1) .. (27,0);
\draw (3,0) .. controls (3,-2) and (7,-2) .. (7,0);
\draw (9,0) .. controls (9,-2) and (13,-2) .. (13,0);
\draw (15,0) .. controls (15,-1) and (16,-1.5) .. (17,-1.5) -- (18,-1.5);
\draw (25,0) .. controls (25,-1) and (24,-1.5) .. (23,-1.5) -- (22,-1.5);
\draw (20.1,-1.5) node {\scalebox{3}{$\cdots$}};
\draw[thick] (11,-1.5) .. controls (13,-1.5) and (9,-4) .. (5,-4)
    .. controls (1,-4) and (-1,-3) .. (0,-2);
\draw[thick,dashed] (0,-2) .. controls (1,-1) and (2,-2.5) .. (5,-2.5)
    .. controls (8,-2.5) and (9,-1.5) .. (11,-1.5);
\draw[thick] (5,-1.5) .. controls (3,-1.5) and (7,-4) .. (11,-4)
    .. controls (15,-4) and (19,-1.5) .. (17,-1.5);
\draw[thick,dashed] (5,-1.5) .. controls (6,-1.5) and (8,-2.5) .. (11,-2.5)
    .. controls (14,-2.5) and (16,-1.5) .. (17,-1.5);
\fill[medgreen] (14,-6) circle (.175);
\fill[medgreen] (11,-4) circle (.175);
\fill[medgreen] (5,-4) circle (.175);
\draw[medgreen] (14,-6) .. controls (10,-6) and (5,-6) .. (5,-4);
\draw[medgreen] (5,-4) .. controls (5,-2.5) and (2,-2) .. (2,-.5);
\draw[medgreen] (5,-4) .. controls (5,-2.5) and (8,-2) .. (8,-.5);
\draw[medgreen] (14,-6) .. controls (12,-6) and (11,-5) .. (11,-4);
\draw[medgreen] (11,-4) .. controls (11,-2.5) and (8,-2) .. (8,-.5);
\draw[medgreen] (11,-4) .. controls (11,-2.5) and (14,-2) .. (14,-.5);
\draw[medgreen] (14,-6) -- (18,-6);
\draw[medgreen] (20.1,-6) node {\scalebox{3}{$\cdots$}};
\draw[medgreen] (22,-6) .. controls (26,-6) and (26,-2) .. (26,-.5);
\draw[medgreen] (14,-6.2) node[anchor=north] {$p_0$};
\draw (2,14) node {$(\Sigma_{g,1})_1$};
\draw (8,14) node {$(\Sigma_{g,1})_2$};
\draw (14,14) node {$(\Sigma_{g,1})_3$};
\draw (26,14) node {$(\Sigma_{g,1})_n$};
\draw[darkblue] (29.5,11) node {$(H_g)_{F,n}$};
\draw[darkred] (22.5,11) node {$(H_g)_{I,n}$};
\draw (-3,3) node {\scalebox{1.5}{$\Sigma_{ng}$}};
\draw (3.5,-4) node[anchor=north east] {$(\Sigma_{2g,1})_{1,2}$};
\draw (14,-3.5) node[anchor=north west] {$(\Sigma_{2g,1})_{2,3}$};
\end{tikzpicture} \end{center}
\caption{The Heegaard surface $\Sigma_{ng}$ with disjoint subsurfaces
$(\Sigma_{g,1})_i$.  Circles that contract in $(H_g)_{I,i} \subset (H_{ng})_I$
are in red, while circles that contract in $(H_g)_{F,i} \subset (H_{ng})_F$
are in blue.  The subsurfaces $(\Sigma_{2g,1})_{i,i+1}$ are also indicated.
The system of basepoints and connecting paths is in green.}
\label{f:heegaard} \end{figure*}

\begin{lemma} Eventually
\[ |R^0_g/K| > |\hR_g \setminus R_g|.\]
\eatline \label{l:outgrow} \end{lemma}

\begin{proof}
Informally, if $g$ is large and we choose a homomorphism $f \in \hR_g$
at random, then it is a surjection with very high probability;
if it is a surjection, then its Schur invariant $\sch(f)$ is
approximately equidistributed.  In detail, Dunfield-Thurston
\cite[Lems.~6.10~\&~6.13]{DT:random} show that
\[ \lim_{g \to \infty} \frac{|R_g|}{|\hR_g|} = 1 \qquad
    \lim_{g \to \infty} \frac{|R^0_g|}{|R_g|} = \frac{1}{|H_2(G)|}. \]
Thus
\[ \lim_{g \to \infty} \frac{|\hR_g \setminus R_g|}{|R^0_g/K|} =
    |H_2(G)|\cdot|K|\cdot\bigg(\lim_{g \to \infty}
    \frac{|\hR_g|}{|R_g|}-1\bigg) = 0. \qedhere \]
\end{proof}

\begin{proof}[Proof of \Thm{th:refine}] We first claim that $\rho_{R^0_g}$
by itself is eventually surjective. Note that the action of $K$ on $R^0_g$
is free; thus we can apply \Thm{th:rubik} if $\rho_{R^0_g}$ satisfies
suitable conditions.  By part 2 of \Thm{th:dt}, $\rho_{R^0_g}$ is eventually
surjective when composed with the quotient $\Rub_K(R^0_g) \to \Sym(R^0_g/K)$.
Meanwhile, part 1 of \Thm{th:dt} says that $\MCG_*(\Sigma_g)$ eventually acts
transitively on $R^0_g(G^2)$.  Since $G$ is simple, \Lem{l:sjoint} tells us
that the homomorphisms $f \in R^0_g(G^2)$ correspond to pairs of surjections
\[ f_1,f_2:\Sigma_g \onto G \]
that are inequivalent under $K = \Aut(G)$.  This eventuality is thus the
condition that the action of $\MCG_*(\Sigma_g)$ is $K$-set 2-transitive
in its action on $R^0_g$.  (\Cf\ Lemma 7.2 in \cite{DT:random}.)
Thus $\rho_{R^0_g}$ eventually satisfies the hypotheses of \Thm{th:rubik}
and is surjective.

The map $\tau$ also surjects $\MCG_*(\Sigma_g)$ onto $\Sp(2g,\Z)$.
Lemmas~\ref{l:symplectic} and \ref{l:rubikquo} thus imply that
$\Rub_K(R^0_g)$ and $\Sp(2g,\Z)$ do not share any simple quotients.
By \Lem{l:zorn}, $\MCG_*(\Sigma_g)$ surjects onto $\Rub_K(R^0_g) \times
\Sp(2g,\Z)$.  Equivalently, $\ker \tau = \Tor_*(\Sigma_g)$ surjects onto
$\Rub_K(R^0_g)$.

Finally we consider 
\[ \rho_{R^0_g} \times \rho_{\hR_g \setminus R_g}:\Tor_*(\Sigma_g)
    \to \Rub_K(R^0_g) \times \Rub_K(\hR_g \setminus R_g), \]
which we have shown surjects onto the first factor. The unique simple
quotient $\Alt(R^0_g/K)$ of $\Rub_K(R^0_g)$ is eventually not involved in
$\Rub_K(\hR_g \setminus R_g)$ because it is too large.  More precisely,
\Lem{l:outgrow} implies that eventually
\[ |\Alt(R^0_g/K)| > |\Alt(\hR_g \setminus R_g)| >
    |\Rub_K(\hR_g \setminus R_g)|. \]
Thus we can apply \Lem{l:zorn} to conclude that the image of
$\Tor_*(\Sigma_g)$ contains $\Rub_K(R^0_g)$, which is equivalent to the
conclusion.
\end{proof}

\subsection{Mapping class gadgets}
\label{ss:mcg}

In this subsection and the next one, we will finish the proof of
\Thm{th:main}.  We want to convert a suitable $\ZSAT$ circuit $C$ of width
$n$ to a homology 3-sphere $M$.  To this end, we choose some sufficiently
large $g$ that depends only on the group $G$, and we let $\Sigma_{ng}$
be a Heegaard surface of $M$.  This Heegaard surface will be decorated in
various ways that we summarize in \Fig{f:heegaard}.  We use the additional
notation that $\Sigma_{g,k}$ is a surface of genus $g$ with $k$ boundary
circles, with a basepoint on one of its circles.  We give $\Sigma_{g,k}$
the representation set
\[ \hR_{g,k} \defeq \{f:\pi_1(\Sigma_{g,k}) \to G\}. \]
We let $\MCG_*(\Sigma_{g,k})$ be the relative mapping class group (that
fixes $\partial \Sigma_{g,k}$); it naturally acts on $\hR_{g,k}$.

We attach two handlebodies $(H_{ng})_I$ and $(H_{ng})_F$ to $\Sigma_{ng}$
so that
\[ (H_{ng})_I \cup (H_{ng})_F \cong S^3.\]
Although an actual sphere $S^3$ is not an interesting homology sphere
for our purposes, our goal is to construct a homeomorphism $\phi \in
\MCG_*(\Sigma_{ng})$ so that
\[ M \defeq (H_{ng})_I \sqcup_\phi (H_{ng})_F \]
is the 3-manifold that we will produce to prove \Thm{th:main}.  (We could
let $\phi$ be an element of the unpointed mapping class group here, but
it is convenient to keep the basepoint.)

\begin{figure}[htb] \begin{center} \begin{tabular}{l|l}
$\ZSAT_{K,A,I,F}$ & $H(M,G)$ \\ \hline
$n$-symbol memory & Heegaard surface $\Sigma_{ng}$ \\
1-symbol memory & computational subsurface $\Sigma_{g,1}$ \\
binary gate & element of $\MCG_*(\Sigma_{2g})$ \\
circuit $C$ & mapping class $\phi \in \MCG_*(\Sigma_{ng})$ \\
alphabet $A$ & homomorphisms $\pi_1(\Sigma_g^1) \to G$\\
zombie symbol: $z \in A$ & trivial map $z:\pi_1(\Sigma_g^1) \to G$ \\
memory state: $x\in A^n$ & homomorphism $f: \pi_1(\Sigma_{ng}) \to G$ \\
initialization: $x \in (I \cup \{z\})^n$ & $f$ extends to $\pi_1((H_{ng})_I)$ \\
finalization: $y \in (F \cup \{z\})^n$ & $f$ extends to $\pi_1((H_{ng})_F)$ \\
solution: $C(x) = y$ & homomorphism $f:\pi_1(M) \to G$
\end{tabular} \end{center}
\caption{A correspondence between $\ZSAT$ and $H(M,G)$.}
\label{f:corresp} \end{figure}

We identify $n$ disjoint subsurfaces
\[ (\Sigma_{g,1})_1,(\Sigma_{g,1})_2,\dots,
    (\Sigma_{g,1})_n \subseteq \Sigma_{ng} \]
which are each homeomorphic to $\Sigma_{g,1}$.  The handlebodies
$(H_{ng})_I$ and $(H_{ng})_F$ likewise have sub-handlebodies $(H_g)_{I,i}$
and $(H_g)_{F,i}$ of genus $g$ associated with $(\Sigma_{g,1})_i$ and
positioned so that
\[ (H_g)_{I,i} \cup (H_g)_{F,i} \cong B^3. \]
We also choose another set of subsurfaces
\[ (\Sigma_{2g,1})_{1,2},(\Sigma_{2g,1})_{2,3},\dots,
    (\Sigma_{2g,1})_{n-1,n} \subseteq \Sigma_{ng} \]
such that
\[ (\Sigma_{g,1})_i,(\Sigma_{g,1})_{i+1} \subseteq (\Sigma_{2g,1})_{i,i+1}. \]
Finally we mark basepoints for each subsurface $(\Sigma_{g,1})_i$ and
$(\Sigma_{2g,1})_{i,i+1}$, and one more basepoint $p_0 \in \Sigma_{ng}$,
and we mark a set of connecting paths as indicated in \Fig{f:heegaard}.

The circuit conversion is summarized in \Fig{f:corresp}.  We will use the
computational alphabet
\[ A \defeq R^0_g \cup \{z\} \subseteq \hR_g \subseteq \hR_{g,1}, \]
where $z:\pi_1(\Sigma_g) \to G$ is (as first mentioned in \Sec{ss:results})
the trivial homomorphism and the zombie symbol, and the inclusion $\hR_g
\subseteq \hR_{g,1}$ comes from the inclusion of surfaces $\Sigma_{g,1}
\subseteq \Sigma_g$.  Note that $\hR^0_g$ is precisely the subset of
$\hR_{g,1}$ consisting of homomorphisms
\[ f:\pi_1(\Sigma_{g,1}) \to G \]
that are trivial on the peripheral subgroup $\pi_1(\partial \Sigma_{g,1})$.

Each subsurface $(\Sigma_{g,1})_i$ is interpreted as the ``memory unit" of
a single symbol $x_i \in A$.  Using the connecting paths in $\Sigma_{ng}$
between the basepoints of its subsurfaces, and since each $x_i$ is trivial
on $\pi_1(\partial \Sigma_{g,1})$, a data register
\[ x = (x_1,x_2,\dots,x_n) \in A^n \]
combines to form a homomorphism
\[ f:\pi_1(\Sigma_{ng}) \to G. \]
In particular, if $x \ne (z,z,\dots,z)$, then $f \in R_{ng}$.  In other
words, $f$ is surjective in this circumstance because one of its components
$x_i$ is already surjective.  (Note that the converse is not true: we can
easily make a surjective $f$ whose restriction to each $(\Sigma_{g,1})_i$
is far from surjective.)

For every subgroup $K \le G$, we define $I(K)$ to be the set of surjections
\[ x:\pi_1(\Sigma_{g,1}) \onto K \]
that come from a homomorphism 
\[ x:\pi_1((H_g)_I) \onto K. \]
We define $F(K)$ in the same way using $(H_g)_F$. A priori we know that
$I(K), F(K) \subseteq R_{g,1}(K)$.  This inclusion can be sharpened in
two significant respects.

\begin{lemma} The sets $I(K)$ and $F(K)$ are subsets of $R^0_g(K)$.
If $K$ is non-trivial, then they are disjoint.
\label{l:ifr0} \end{lemma}

\begin{proof} First, since $\partial \Sigma_{g,1}$ bounds a disk in
$(H_g)_I$, we obtain that $I(K), F(K) \subseteq R_g(K)$.  Second,
since any $x$ in $I(K)$ or $F(K)$ extends to a handlebody, the cycle
$x_*([\Sigma_g])$ is null-homologous in $BG$ and therefore $\sch(x) =
0.$  Third, since $(H_g)_I \cup (H_g)_F \cong B^3$ is simply connected, a
surjective homomorphism $x \in R_g(K)$ cannot extend to both handlebodies if
$K$ is non-trivial.  Therefore $I(K)$ and $F(K)$ are disjoint in this case.
\end{proof}

The gadgets that serve as binary gates are mapping class elements $\alpha
\in \MCG_*((\Sigma_{2g,1})_{i,i+1})$ that act on two adjacent memory units
$(\Sigma_{g,1})_i$ and $(\Sigma_{g,1})_{i+1}$.  We summarize the effect
of the local subsurface inclusions on representation sets.  In order to
state it conveniently, if $X$ and $Y$ are two pointed spaces, we define a
modified wedge $X \vee_\lambda Y$, where $\lambda$ is a connecting path
between the basepoint of $X$ and the basepoint of $Y$.  \Fig{f:pinched}
shows a surjection from $\Sigma_{2g}$ to $\Sigma_g \vee_\lambda \Sigma_g$,
while \Fig{f:heegaard} has copies of $\Sigma_{g,1} \vee_\lambda \Sigma_{g,1}$
(which has a similar surjection from $\Sigma_{2g,1}$.

\begin{figure*}[htb] \begin{center}
\begin{tikzpicture}[semithick,scale=.4]
\draw (0,-1) .. controls (1,-1) and (2,-2) .. (3,-2)
    .. controls (4,-2) and (5,-1) .. (7,-1)
    .. controls (9,-1) and (10,-2) .. (11,-2)
    .. controls (12,-2) and (13,-1) .. (13,0)
    .. controls (13,1) and (12,2) .. (11,2)
    .. controls (10,2) and (9,1) .. (7,1)
    .. controls (5,1) and (4,2) .. (3,2)
    .. controls (2,2) and (1,1) .. (0,1)
    .. controls (-1,1) and (-2,2) .. (-3,2)
    .. controls (-4,2) and (-5,1) .. (-7,1)
    .. controls (-9,1) and (-10,2) .. (-11,2)
    .. controls (-12,2) and (-13,1) .. (-13,0)
    .. controls (-13,-1) and (-12,-2) .. (-11,-2)
    .. controls (-10,-2) and (-9,-1) .. (-7,-1)
    .. controls (-5,-1) and (-4,-2) .. (-3,-2)
    .. controls (-2,-2) and (-1,-1) .. (0,-1);
\draw[thick] (0,1) arc (90:270:.5 and 1);
\draw[thick,dashed] (0,-1) arc (-90:90:.5 and 1);
\foreach \x in {-11,-3,3,11} {
\begin{scope}[shift={(\x,0)}]
\draw (-.8,0) .. controls (-.4,.4) and (.4,.4) .. (.8,0);
\draw (-.9,.1) -- (-.8,0) .. controls (-.4,-.4) and (.4,-.4) .. (.8,0)
    -- (.9,.1);
\end{scope} }
\draw (7.1,-.05) node {\scalebox{3}{$\cdots$}};
\draw (-6.9,-.05) node {\scalebox{3}{$\cdots$}};
\foreach \x in {-8,8} {
\begin{scope}[shift={(\x,-9)}]
\draw (-6,0) .. controls (-6,-1) and (-5,-2) .. (-4,-2)
    .. controls (-3,-2) and (-2,-1) .. (0,-1)
    .. controls (2,-1) and (3,-2) .. (4,-2)
    .. controls (5,-2) and (6,-1) .. (6,0)
    .. controls (6,1) and (5,2) .. (4,2)
    .. controls (3,2) and (2,1) .. (0,1)
    .. controls (-2,1) and (-3,2) .. (-4,2)
    .. controls (-5,2) and (-6,1) .. (-6,0);
\draw (.1,-.05) node {\scalebox{3}{$\cdots$}};
\end{scope} }
\foreach \x in {-12,-4,4,12} {
\begin{scope}[shift={(\x,-9)}]
\draw (-.8,0) .. controls (-.4,.4) and (.4,.4) .. (.8,0);
\draw (-.9,.1) -- (-.8,0) .. controls (-.4,-.4) and (.4,-.4) .. (.8,0)
    -- (.9,.1);
\end{scope} }
\draw (-2,-9) -- (2,-9); \draw (0,-9) node[anchor=north] {$\lambda$};
\fill (-2,-9) circle (.175); \fill (2,-9) circle (.175);
\draw (-15,0) node[anchor=east] {\scalebox{1.25}{$\Sigma_{2g}$}};
\draw (-15,-9) node[anchor=east]
    {\scalebox{1.25}{$\Sigma_g \vee_\lambda \Sigma_g$}};
\draw (0,-4.5) node {\scalebox{2}{\rotatebox{270}{
    $\relbar\joinrel\twoheadrightarrow$}}};
\end{tikzpicture}
\end{center}
\caption{From $\Sigma_{2g}$ to $\Sigma_g \vee_\lambda \Sigma_g$}
\label{f:pinched} \end{figure*}
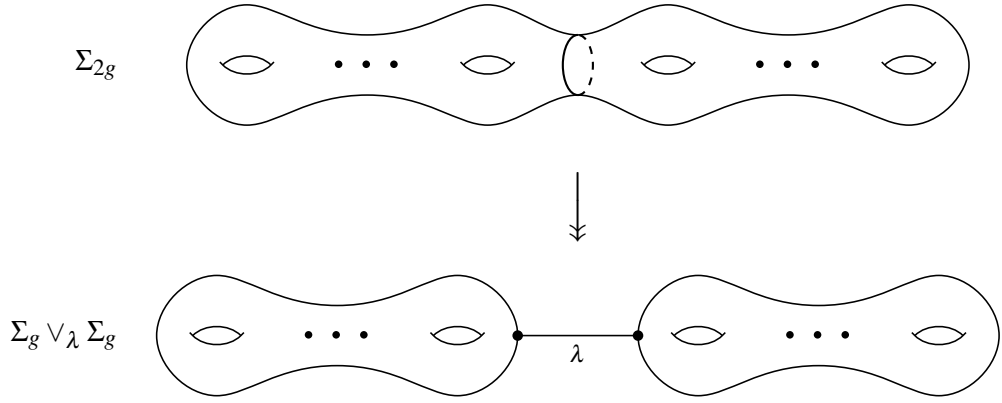

\begin{lemma} The inclusions and surjections
\[ \begin{array}{ccc}
\Sigma_{2g,1} & \subseteq & \Sigma_{2g} \\[-1.5ex]
\vonto  & & \vonto \\[2ex]
\Sigma_{g,1} \vee_\lambda \Sigma_{g,1} & \subseteq
    & \Sigma_g \vee_\lambda \Sigma_g 
\end{array} \]
yield the inclusions
\begin{eq}{e:include1} \begin{array}{ccccccc}
\hR_{2g,1} & \supseteq & \hR_{2g} & \supseteq & R_{2g}
    & \supseteq & R^0_{2g} \\[.5ex]
\veq & & \vsubseteq & & \vsubseteq & & \vsubseteq \\
\hR_{g,1} \times \hR_{g,1} & \supseteq & \hR_g \times \hR_g & \supseteq &
    R_g \times R_g & \supseteq & R^0_g \times R^0_g
\end{array}. \end{eq}
For every pair of subgroups $K_1, K_2 \le G$ that generate $K \le G$,
they also yield
\begin{eq}{e:include2} R^0_g(K_1) \times R^0_g(K_2)
    \subseteq R^0_{2g}(K). \end{eq}
Finally, they yield
\begin{eq}{e:include3} A \times A \subseteq R^0_{2g}
    \cup \{z_{2g}\}, \end{eq}
where $z_g \in R_g$ is the trivial map in genus $g$ and $z_{2g} = (z_g,z_g)$.
\eatline \label{l:include} \end{lemma}

\begin{proof} The horizontal inclusions are all addressed above; the real
issue is the vertical inclusions and equalities.   We consider the vertical
inclusions from left to right in diagram \eqref{e:include1}.  The surjection
\[ \sigma_1:\Sigma_{2g,1} \onto \Sigma_{g,1} \vee_\lambda \Sigma_{g,1} \]
is an isomorphism of $\pi_1$, while the surjection
\[ \sigma_0:\Sigma_{2g} \onto \Sigma_g \vee_\lambda \Sigma_g \]
is a surjection in $\pi_1$.  This implies the first two vertical
relations.  Then, if two homomorphisms
\[ f_1,f_2:\pi_1(\Sigma_g) \onto G \]
are each surjective, then they are certainly jointly surjective; this
implies the third relation.  Finally, the surjection $\sigma_0$
yields the formula 
\begin{eq}{e:sch} \sch((f_1,f_2)) = \sch(f_1) + \sch(f_2). \end{eq}
The reason is that the image $\sigma_0([\Sigma_{2g}])$ of the fundamental
class of $\Sigma_{2g}$ is the sum of the fundamental classes of the 
two $\Sigma_g$ components.  This yields the fourth, leftmost inclusion
because equation \eqref{e:sch} then reduces to $0 = 0+0$.

To treat \eqref{e:include2}, we claim that if $\sch_{K_i}(f_i) = 0$,
then $\sch_K(f_i) = 0$.  This follows from the fact that each map from
$\Sigma_g$ to the classifying space $BK_i$ and $BK$ forms a commutative
triangle with the map $BK_i \to BK$.  With this remark, inclusion
\eqref{e:include2} can be argued in the same way as the inclusion $R^0_g
\times R^0_g \subseteq R^0_{2g}$.

Finally for inclusion \eqref{e:include3}, recall that $A = R^0_g \cup
\{z_g\}$, and that $z_{2g} = (z_g,z_g)$ since in each case $z$ is the
trivial homomorphism.  The inclusions
\[ R^0_g \times \{z_g\}, \{z_g\} \times R^0_g \subseteq R^0_{2g} \]
can be argued the same way as before: Given the two homomorphisms $f_1,f_2$,
even if one of them is the trivial homomorphism $z_g$, the surjectivity
of the other one gives us joint surjectivity.  Moreover, the trivial
homomorphism $z_g$ has a vanishing Schur invariant $\sch_G(z_g) = 0$
relative to the target group $G$.
\end{proof}

\subsection{End of the proof}
\label{ss:end}

We combine \Thm{th:refine} with Lemmas \ref{l:include} and \ref{l:zsat}
to convert a circuit $C$ in $\ZSAT_{K,A,I,F}$ to a mapping class $\phi \in
\MCG_*(\Sigma_{ng})$ using mapping class gadgets.  To apply \Lem{l:zsat},
we need to verify the conditions in \eqref{e:zineq}.  These follow
easily from asymptotic estimates on the cardinality of $A$ and $I$
\cite[Lems.~6.10~\&~6.11]{DT:random}.

For each $\gamma \in \Rub_K(A \times A)$, we choose an $\alpha \in
\Tor_*(\Sigma_{2g,1})$ such that:
\begin{enumerate}
\item $\alpha$ acts by $\gamma$ on $A \times A$.
\item $\alpha$ acts by an element of $\Rub_K(R^0_{2g})$
that fixes $R^0_{2g} \setminus (A \times A)$.
\item $\alpha$ fixes $\hR_{2g} \setminus R_{2g}$.
\end{enumerate}
Given a circuit $C$ in $\ZSAT_{K,A,I,F}$, we can replace each
gate $\gamma \in \Rub_K(A \times A)$ that acts on symbols $i$
and $i+1$ by the corresponding local mapping class $\alpha \in
\Tor_*((\Sigma_{2g,1})_{(i,i+1)})$.  Then we let $\phi$ be the composition
of the gadgets $\alpha$.

\begin{lemma} Let
\[ M \defeq (H_{ng})_I \sqcup_\phi (H_{ng})_F. \]
Then
\begin{enumerate}
\item $M$ is a homology 3-sphere.
\item If $1 \lneq K \lneq G$ is a non-trivial, proper subgroup of $G$,
then $Q(M,K) = \emptyset$.
\item $\#H(M,G) = \#C.$
\end{enumerate}
\label{l:glue} \end{lemma}

\begin{proof} Point 1 holds because by construction,
$\phi \in \Tor(\Sigma_{2g})$.

To address points 2 and 3, we decompose $\phi$ as a composition of 
local gadgets,
\begin{eq}{e:phi} \phi = \alpha_m \circ \alpha_{m-1} \circ \dots \circ
    \alpha_2 \circ \alpha_1, \end{eq}
and we insert parallel copies $(\Sigma_{ng})_j$ of the Heegaard surface
with $0 \le j \le m$, so the $i$th gadget $\alpha_j$ yields a map
\[ \alpha_j:(\Sigma_{ng})_{j-1} \to (\Sigma_{ng})_j \]
from the $(j-1)$-st to the $j$-th surface.  Each $\alpha_j$ is a non-trivial
homeomorphism
\[ \alpha_j:(\Sigma_g)_{j-1,(i,i+1)} \to (\Sigma_{ng})_{j,(i,i+1)} \]
for some $i$, and is the identity elsewhere.  We use this decomposition
to analyze the possibilities for a group homomorphism
\[ f:\pi_1(M) \to G. \]
The map $f$ restricts to a homomorphism
\[ f_j:\pi_1((\Sigma_{ng})_j) \to G, \]
and then further restricts to a homomorphism
\[ f_{j,i}:\pi_1((\Sigma_{g,1})_{j,i}) \to G \]
for the $i$th memory unit for each $i$.   It is convenient to interpret
$\hR_{g,1} \supseteq A$ as the superalphabet of all possible symbols that
could in principle arise as the state of a memory unit.

By construction, each initial symbol $f_{0,i}$ extends to the handlebody
$(H_g)_{I,i}$.   Thus $f_{0,i} \in I(K)$ for some subgroup $1 \le K \le
G$, and all cases are disjoint from $A$ other than $K = 1$ and $K=G$.
Likewise at the end, each $f_{m,i} \in F(K)$ for some $K$.  By construction,
each $\alpha_j$ fixes both $R^0_{2g} \setminus (A \times A)$ and $\hR_{2g}
\setminus R_{2g}$.  This fixed set includes all cases $R^0(K_1) \times
R^0(K_2)$, and therefore all cases $I(K_1) \times I(K_2)$, other than
$K_1,K_2 \in \{1,G\}$.  Thus every initial symbol $f_{0,i} \in I(K)
\not\subseteq A$ is preserved by every gadget $\alpha_j$, and then can't
finalize because $I(K) \cap F(K) = \emptyset$.  Among other things, this
establishes point 2 of the lemma.

This derivation also restricts the initial state $f_0$ to $A^n$.  In this
case, each $\alpha_j$ acts in the same way on $A^n$ as the corresponding gate
$\gamma_j$; consequently it leaves the set $A^n$ invariant.  Considering both
the circuit action and initialization and finalization, these states exactly
match the behavior of the circuit $C$ under the rules of $\ZSAT_{K,A,I,F}$.
\end{proof}

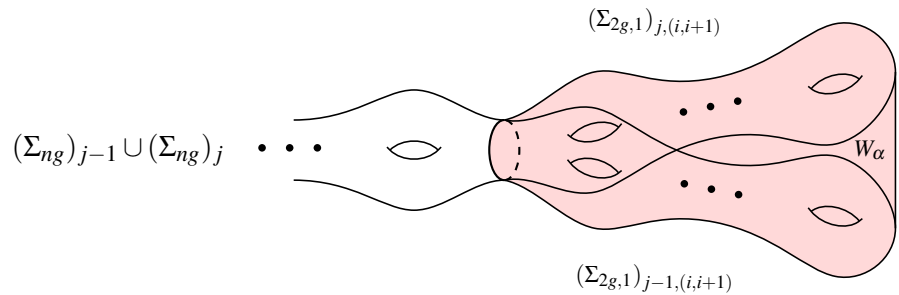
\begin{figure*}[htb] \begin{center}
\begin{tikzpicture}[semithick,scale=.4,even odd rule]
\fill[white!85!red] (0,-1) .. controls (1,-1.2) and (2,-2.4) .. (3,-2.6)
    .. controls (4,-2.8) and (5,-2) .. (7,-2.4)
    .. controls (9,-2.8) and (10,-4) .. (11,-4.2)
    .. controls (12,-4.4) and (13,-3.6) .. (13,-2.6) -- (13,2.6)
    .. controls (13,3.6) and (12,4.4) .. (11,4.2)
    .. controls (10,4) and (9,2.8) .. (7,2.4)
    .. controls (5,2) and (4,2.8) .. (3,2.6)
    .. controls (2,2.4) and (1,1.2) .. (0,1)
    arc  (90:270:.5 and 1);
\draw (13,-2.6) -- (13,2.6);
\draw (0,-1) .. controls (1,-.8) and (2,-1.6) .. (3,-1.4)
    .. controls (4,-1.2) and (5,0) .. (7,.4)
    .. controls (9,.8) and (10,0) .. (11,.2)
    .. controls (12,.4) and (13,1.6) .. (13,2.6)
    .. controls (13,3.6) and (12,4.4) .. (11,4.2)
    .. controls (10,4) and (9,2.8) .. (7,2.4)
    .. controls (5,2) and (4,2.8) .. (3,2.6)
     .. controls (2,2.4) and (1,1.2) .. (0,1);
\draw (2.2,.44) .. controls (2.6,.92) and (3.4,1.08) .. (3.8,.76);
\draw (2.1,.52) -- (2.2,.44) .. controls (2.6,.12) and (3.4,.28)
    .. (3.8,.76) -- (3.9,.88);
\draw (10.2,2.04) .. controls (10.6,2.52) and (11.4,2.68) .. (11.8,2.36);
\draw (10.1,2.12) -- (10.2,2.04) .. controls (10.6,1.72) and (11.4,1.88)
    .. (11.8,2.36) -- (11.9,2.48);
\draw (7,1.4) node {\rotatebox{11.3}{\scalebox{3}{$\cdots$}}};
\draw (0,-1) .. controls (1,-1.2) and (2,-2.4) .. (3,-2.6)
    .. controls (4,-2.8) and (5,-2) .. (7,-2.4)
    .. controls (9,-2.8) and (10,-4) .. (11,-4.2)
    .. controls (12,-4.4) and (13,-3.6) .. (13,-2.6)
    .. controls (13,-1.6) and (12,-.4) .. (11,-.2)
    .. controls (10,0) and (9,-.8) .. (7,-.4)
    .. controls (5,0) and (4,1.2) .. (3,1.4)
     .. controls (2,1.6) and (1,.8) .. (0,1);
\draw (2.2,-.44) .. controls (2.6,-.12) and (3.4,-.28) .. (3.8,-.76);
\draw (2.1,-.32) -- (2.2,-.44) .. controls (2.6,-.92) and (3.4,-1.08)
    .. (3.8,-.76) -- (3.9,-.68);
\draw (10.2,-2.04) .. controls (10.6,-1.72) and (11.4,-1.88) .. (11.8,-2.36);
\draw (10.1,-1.92) -- (10.2,-2.04) .. controls (10.6,-2.52) and (11.4,-2.68)
    .. (11.8,-2.36) -- (11.9,-2.28);
\draw (7,-1.4) node {\rotatebox{-11.3}{\scalebox{3}{$\cdots$}}};
\draw (0,1) .. controls (-1,1) and (-2,2) .. (-3,2)
    .. controls (-4,2) and (-5,1) .. (-7,1);
\draw (-7,-1) .. controls (-5,-1) and (-4,-2) .. (-3,-2)
    .. controls (-2,-2) and (-1,-1) .. (0,-1);
\draw (-2.2,0) .. controls (-2.6,.4) and (-3.4,.4) .. (-3.8,0);
\draw (-2.1,.1) -- (-2.2,0) .. controls (-2.6,-.4) and (-3.4,-.4) .. (-3.8,0)
    -- (-3.9,.1);
\draw (-7,0) node {\scalebox{3}{$\cdots$}};
\draw[thick] (0,1) arc (90:270:.5 and 1);
\draw[thick,dashed] (0,-1) arc (-90:90:.5 and 1);
\draw (13,0) node[anchor=east] {$W_\alpha$};
\draw (5,-3.5) node[anchor=north] {$(\Sigma_{2g,1})_{j-1,(i,i+1)}$};
\draw (5,3.5) node[anchor=south] {$(\Sigma_{2g,1})_{j,(i,i+1)}$};
\draw (-9,0) node[anchor=east] {\scalebox{1.25}{
    $(\Sigma_{ng})_{j-1} \cup (\Sigma_{ng})_j$}};
\end{tikzpicture}
\end{center}
\caption{The blister $W_\alpha$ between $(\Sigma_{2g,1})_{j-1,(i,i+1)}$ and
    $(\Sigma_{2g,1})_{j,(i,i+1)}$.}
\label{f:blister} \end{figure*}

To complete the proof of \Thm{th:main}, we only need to efficiently
triangulate the 3-manifold $M \defeq (H_{ng})_I \sqcup_\phi (H_{ng})_F$.
The first step is to refine the decoration of $\Sigma_{ng}$ shown in
\Fig{f:heegaard} to a triangulation.   It is easy to do this with polynomial
complexity in $n$ (or in $ng$, but recall that $g$ is fixed).  We can also
give each subsurface $(\Sigma_{g,1})_i$ thesame triangulation for all $i$,
as well as each subsurface $(\Sigma_{2g,1})_{i,i+1}$.  It is also routine
to extend any such triangulation to either $(H_{ng})_I$ or $(H_{ng})_F$ with
polynomial (indeed linear) overhead:  Since by construction the triangulation
of each $(\Sigma_{g,1})_i$ is the same, we pick some extension to $(H_g)_I$
and $(H_g)_F$ and use it for each $(H_g)_{I,i}$ and each $(H_g)_{F,i}$.
The remainder of $(H_{ng})_I$ and $(H_{ng})_F$ is a 3-ball whose boundary
has now been triangulated; any triangulation of the boundary of a 3-ball can
be extended to the interior algorithmically and with polynomial complexity.

We insert more triangulated structure in between $(H_{ng})_I$ and
$(H_{ng})_F$ to realize the homeomorphism $\phi$.  Recalling equation
\eqref{e:phi} in the proof of \Lem{l:glue}, $\phi$ decomposes into
local mapping class gadgets $\alpha_j$.  Only finitely many $\alpha \in
\MCG_*(\Sigma_{g,1})$ are needed, since we only need one representative
for each $\gamma \in \Rub_K(A \times A)$.  At this point it is convenient
to use a blister construction.  We make a 3-manifold $W_\alpha$ whose
boundary is two copies of $\Sigma_{2g,1}$ (with its standard triangulation)
that meet at their boundary circle, and so that $W_\alpha$ is a relative
mapping cylinder for the homeomorphism $\alpha$.   If $\alpha_j$ acts
on $(\Sigma_{2g,1})_{i,i+1}$, then we can have $(\Sigma_{ng})_{j-1}$ and
$(\Sigma_{ng})_j$ coincide outside of $(\Sigma_{2g,1})_{j-1,(i,i+1)}$ and
$(\Sigma_{2g,1})_{j,(i,i+1)}$, so that their union $(\Sigma_{ng})_{j-1}
\cup (\Sigma_{ng})_j$ is a branched surface.  We insert $W_\alpha$
and its triangulation in the blister within $(\Sigma_{ng})_{j-1} \cup
(\Sigma_{ng})_j$; see \Fig{f:blister}.

\section{Final remarks and questions}
\label{s:final}

\subsection{Sharper hardness}
\label{ss:sharper}

Even though the proof of \Thm{th:main} is a polynomially efficient reduction,
for any fixed, suitable target group $G$, it is not otherwise particularly
efficient.   Various steps of the proof require the genus $g$ (which is used
to define the symbol alphabet $Y_g^0$) to be sufficiently large.  In fact,
the crucial \Thm{th:dt} does not even provide a constructive lower bound
on $g$.  Dunfield and Thurston \cite{DT:random} discuss possibilities
to improve this bound, and they conjecture that $g \ge 3$ suffices in
\Thm{th:dt} for many or possibly all choices of $G$.  We likewise believe
that there is some universal genus $g_0$ such that \Thm{th:refine} holds
for all $g \ge g_0$.

In any case, the chain of reduction summarized in \Fig{f:reductions}
is not very efficient either.  What we really believe is that the random
3-manifold model of Dunfield and Thurston also yields computational hardness.
More precisely, Johnson showed that the Torelli group $\Tor(\Sigma_g)$
is finitely generated for $g \ge 3$ \cite{Johnson:finite}.  This yields
a model for generating a random homology 3-sphere:  We choose $\phi \in
\Tor(\Sigma_g)$ by evaluating a word of length $\ell$ in the Johnson
generators, and then we let
\[ M \defeq (H_g)_I \sqcup_\phi (H_g)_F. \]
Our \Thm{th:refine} implies that \cite[Thm.~7.1]{DT:random} holds in this
model, \ie, that the distribution of $\#Q(M,G)$ converges to Poisson with
mean $|H_2(G)|/|\Out(G)|$ if we first send $\ell \to \infty$ and then send
$g \to \infty$.  We also conjecture that $\#Q(M,G)$ is hard on average in
the sense of average-case computational complexity \cite[Ch.~18]{AB:modern}
if $\ell$ grows faster than $g$.

Speaking non-rigorously, we conjecture that it is practical to randomly
generate triangulated homology 3-spheres $M$ in such a way that no one
will ever know the value of $\#Q(M,G)$, say for $G = A_5$.  Hence, no one
will ever know whether such an $M$ has a connected 5-sheeted cover.

\subsection{Other spaces}
\label{ss:other}

Maher \cite{Maher:heegaard} showed that the probability that a
randomly chosen $M$ in the Dunfield-Thurston model is hyperbolic converges
to 1 as $\ell \to \infty$, for any fixed $g \ge 2$.  Maher notes that the
same result holds if $M$ is a homology 3-sphere made using the Torelli
group, for any $g \ge 3$.  Thus our conjectures in \Sec{ss:sharper}
would imply that $\#Q(M,G)$ is computationally intractible
when $M$ is a hyperbolic homology 3-sphere.

We conjecture that a version of \Thm{th:main} holds when $M$ fibers over
a circle.  In this case $M$ cannot be a homology 3-sphere, but it can be a
homology $S^2 \times S^1$.  If $M$ fibers over a circle, then the invariant
$H(M,G)$ is obviously analogous (indeed a special case of) counting solutions
to $C(x) = x$ when $C$ is a reversible circuit.  However, the reduction from
$C$ to $M$ would require new techniques to avoid spurious solutions.

In forthcoming work \cite{K:coloring}, we will prove an analogue of
\Thm{th:main} when $M$ is a knot complement.  We will use a theorem of
Roberts and Venkatesh \cite{RV:hurwitz} which is itself an analogue of
\Thm{th:dt} for braid group actions.

\subsection{Non-simple groups}
\label{ss:nonsimple}

We consider the invariant $\#H(M,G)$ for a general finite group $G$.

Recall that the \emph{perfect core} $G_\per$ of a group $G$ is its unique
largest perfect subgroup; if $G$ is finite, then it is also the limit of its
derived series.   If $M$ is a homology sphere, then its fundamental group
is perfect and $H(M,G) = H(M,G_\per)$.  We conjecture then that a version
of \Thm{th:main} holds for any finite, perfect group $G$.  More precisely,
we conjecture that \Thm{th:main} holds for $Q(M,G)$ when $G$ is finite
and perfect, and that the rest of $H(M,G)$ is explained by non-surjective
homomorphisms $f:G \to G$. Mochon's analysis \cite{Mochon:finite} in the
case when $G$ is non-solvable can be viewed as a partial result towards
this conjecture.

If $G$ is finite and $G_\per$ is trivial, then this exactly the case that
$G$ is solvable.  In the case when $M$ is a link complement, Ogburn and
Preskill \cite{OP:topological} non-rigorously conjecture that $H(M,G)$
is not ``universal" for classical computation.  It is very believable
that the relevant actions of braid groups and mapping class groups are
too rigid for any analogue of the second half of \Thm{th:dt} to hold.
Rowell \cite{Rowell:paradigms} more precisely conjectured that $\#H(M,G)$
can be computed in polynomial time for any link complement $M$ and any
finite, solvable $G$.  We are much less confident that this more precise
conjecture is true.


\providecommand{\bysame}{\leavevmode\hbox to3em{\hrulefill}\thinspace}
\providecommand{\MR}{\relax\ifhmode\unskip\space\fi MR }
\providecommand{\MRhref}[2]{%
  \href{http://www.ams.org/mathscinet-getitem?mr=#1}{#2}
}
\providecommand{\href}[2]{#2}
\providecommand{\eprint}{\begingroup \urlstyle{tt}\Url}

\typeout{get arXiv to do 4 passes: Label(s) may have changed. Rerun}

\end{document}